\documentclass[jsl]{amsart}

\usepackage{amssymb}
\usepackage{algorithm2e}
\usepackage[colorlinks=true]{hyperref}

\title{A comparison of various analytic choice principles}

\author{Paul-Elliot Angl\`es d'Auriac}
\address{LACL, D\'epartement d'Informatique\\
Facult\'e des Sciences et Technologie\\
61 avenue du G\'en\'eral de Gaulle\\
94010 Cr\'eteil Cedex, France}

\email{panglesd@lacl.fr}

\thanks{Angl\`es d'Auriac would like to thank the JSPS, as
  the paper was prepared during Summer Program of the Japan Society
  for the Promotion of Science.}

\author{Takayuki Kihara}
\address{Department of Mathematical Informatics\\
  Graduate School of Informatics\\
  Nagoya University, Japan}
\email{kihara@i.nagoya-u.ac.jp}

\thanks{Kihara's research was partially supported by JSPS KAKENHI
  Grant 17H06738, 15H03634, and the JSPS Core-to-Core Program
  (A. Advanced Research Networks).}

\newtheorem{Theorem}{Theorem}[section]
\newtheorem{Proposition}[Theorem]{Proposition}%[section]
\newtheorem{Corollary}[Theorem]{Corollary}%[section]
\newtheorem{Lemma}[Theorem]{Lemma}%[section]
\theoremstyle{definition}
\newtheorem{Definition}[Theorem]{Definition}
\newtheorem{Fact}[Theorem]{Fact}

\newtheorem{Obs}[Theorem]{Observation}
\newtheorem*{Notation}{Notation}
\newtheorem*{Ack}{Acknowledgements}
\newtheorem{Question}[Theorem]{Question}%[section]

\newcommand{\rest}[1]{\! \upharpoonright_{#1}} 

\newcommand\Nb{\mathbb{N}}
\newcommand\dom{\mathrm{dom}}
\newcommand\omCK{\omega_1^{\mathrm{CK}}}
\newcommand\Baire{{\omega^{\omega}}}
\newcommand\baire{{\omega^{<\omega}}}
\newcommand\Cantor{{2^{\omega}}}

\newcommand\sW{_{\rm sW}}
\newcommand\W{_{\rm W}}

\newcommand\N{\mathbb{N}}
\newcommand\upto{\upharpoonright}
\newcommand\ac[1]{\Sigma^1_1\mbox{-}{\sf AC}^{\sf #1}_{\N^\N}}
\newcommand\dc[1]{\Sigma^1_1\mbox{-}{\sf DC}^{\sf #1}_{\N^\N}}
\newcommand{\pcolon}{\colon\!\!\!\subseteq}
\newcommand\om{\omega}
\newcommand{\fr}{\mbox{}^\smallfrown}

\begin{document}

\begin{abstract}
We investigate computability theoretic and descriptive set theoretic contents of various kinds of analytic choice principles by performing detailed analysis of the Medvedev lattice of $\Sigma^1_1$-closed sets.
Among others, we solve an open problem on the Weihrauch degree of the parallelization of the $\Sigma^1_1$-choice principle on the integers.
Harrington's unpublished result on a jump hierarchy along a pseudo-well-ordering plays a key role in solving the problem.
\end{abstract}

\maketitle

\section{Introduction}

\subsection{Summary}

The study of the Weihrauch lattice aims to measure the computability theoretic difficulty of finding a choice function witnessing the truth of a given $\forall\exists$-theorem (cf.\ \cite{pauly-handbook}) as an analogue of reverse mathematics \cite{SOSOA:Simpson}.
In this article, we investigate the uniform computational contents of the axiom of choice $\Sigma^1_1$-{\sf AC} and dependent choice $\Sigma^1_1$-{\sf DC} for $\Sigma^1_1$ formulas in the context of the Weihrauch lattice.

The computability-theoretic strength of these choice principles is completely independent of their proof-theoretic strength, since the meaning of an impredicative notion such as $\Sigma^1_1$ is quite unstable among models of second-order arithmetic.
Nevertheless, it is still interesting to examine the uniform computational contents of $\Sigma^1_1$-{\sf AC} and $\Sigma^1_1$-{\sf DC} in the full model $\mathcal{P}\N$:
In descriptive set theory, we do not consider the complexity of points in spaces.
Instead, we consider the descriptive or topological complexity of sets and functions on spaces as described below.

For a set $A\subseteq X\times Y$ define the {\em $x$-th section of $A$} as $A(x)=\{y\in Y:(x,y)\in A\}$.
Moreover, we say that a set is {\em total} if all of its sections are nonempty.
We say that a partial function $g\pcolon X\to Y$ is a {\em choice function} for $A$ if $g(x)$ is defined and $g(x)\in A(x)$ whenever $A(x)$ is nonempty.
In descriptive set theory and related areas, there are a number of important results on measuring the complexity of choice functions.
Let $X$ and $Y$ be standard Borel spaces.
The Jankov-von Neumann uniformization theorem (cf.\ \cite[Theorem 18.1]{KechrisBook}) states that if $A$ is analytic, then there is a choice function for $A$ which is measurable w.r.t.\ the $\sigma$-algebra generated by the analytic sets.
The Luzin-Novikov uniformization theorem (cf.\ \cite[Theorem 18.10]{KechrisBook}) states that if $A$ is Borel each of whose section is at most countable, then there is a Borel-measurable choice function for $A$.
Later, Arsenin and Kunugui (cf.\ \cite[Theorem 35.46]{KechrisBook}) showed that the same holds even if each section is allowed to be $\sigma$-compact.

A set $H\subseteq Z^Y$ is {\em homogeneous} if $H$ is the set of all total choice functions for some $A\subseteq Y\times Z$.
A choice function for a set with homogeneous sections can be thought of as a choice of a choice function.
The fact that the coanalytic sets do not have the separation property can be used to conclude that an analytic set with compact homogeneous sections does not necessarily have a Borel-measurable choice.
Nevertheless, a set with homogeneous sections is sometimes easier to uniformize than a general set.
For instance, a coanalytic subset of $X\times \om^\om$ with homogeneous sections always have a Borel-measurable choice, whereas there is no complexity bound within $\mathbf{\Delta}^1_2$ which has a power to uniformize a coanalytic set even if assuming that every section is a singleton.

We are interested in comparing the difficulty of finding choice functions for various analytic sets.
Our main tools for comparing the degrees of difficulty are the following preorderings on analytic sets in product spaces.
Let $A\subseteq X\times Y$ and $B\subseteq Z\times W$ be given.
\begin{enumerate}
\item We write $A\leq_1 B$ if there exist continuous functions $h\colon X\to W$ and $k:Z\to Y$ such that $k\circ g\circ h$ is a choice for $A$ whenever $g$ is a choice for $B$.
\item We write $A\leq_2 B$ if there exist continuous functions $h\colon X\to W$ and $k\colon Z\to Y$ such that $x\mapsto k(x,g\circ h(x))$ is a choice for $A$ whenever $g$ is a choice for $B$.
\end{enumerate}

It is clear that $A\leq_1B$ always implies $A\leq_2B$, but the converse does not hold in general.
Note that $\leq_0$ usually refers the Wadge reducibility, and the two preorderings $\leq_1$ and $\leq_2$ are topological versions of two reducibility notions $\leq\sW$ and $\leq\W$ introduced in Section \ref{sec:prelim}.

\begin{Fact}[Kihara-Marcone-Pauly \cite{KMP}]
For any total analytic set $A\subseteq\om^\om\times 2^\om$, there exists a total analytic set $H\subseteq\om^\om\times 2^\om$ with homogeneous sections such that $A\leq_1 H$.

However, there exists a total analytic set $A\subseteq\om^\om\times \om^\om$ with homogeneous sections such that $A\not\leq_2 B$ for any total analytic set $B\subseteq\om^\om\times \om^\om$ with compact sections.
\end{Fact}

\begin{Question}[Brattka et al.\ \cite{paulybrattka5} and Kihara et al.\ \cite{KMP}]\label{que:Borel-choice}
For any total analytic set $A\subseteq\om^\om\times\om^\om$, does there exist a total analytic set $H\subseteq\om^\om\times\om^\om$ with homogeneous sections such that $A\leq_2 H$?
\end{Question}

In this article, we compare the complexity of choice principles for various kinds of analytic sets, that is, analytic sets with compact sections, $\sigma$-compact sections, homogeneous sections, and so on.
In particular, we negatively solve Question \ref{que:Borel-choice}.

To solve this question, we will employ the notion of a pseudo-hierarchy:
A remarkable discovery by Harrison is that some {\em non}-well-ordering $\prec$ admits a transfinite recursion based on an arithmetical formula.
Furthermore, a basic observation is that, without deciding if a given countable linear ordering $\prec$ is well-ordered or not, one can either proceed an arithmetical transfinite recursion along $\prec$ or construct an infinite $\prec$-decreasing sequence.
Indeed, we will see that the degree of difficulty of such a construction is quite close to that of uniformizing analytic sets with compact sections, which is drastically easier than deciding well-orderedness of a countable linear ordering.

\subsection{Preliminaries}\label{sec:prelim}

In this article, we investigate several variants of $\Sigma^1_1$-choice principles in the context of the Weihrauch lattice.
The notion of Weihrauch degree is used as a tool to classify certain $\forall\exists$-statements by identifying $\forall\exists$-statements with a partial multivalued function.
Informally speaking, a (possibly false) statement $S\equiv\forall x\in X\;[Q(x)\;\rightarrow\;\exists x\;P(x,y)]$ is transformed into a partial multivalued function $f\pcolon X\rightrightarrows Y$ such that ${\rm dom}(f)=\{x:Q(x)\}$ and $f(x)=\{y:P(x,y)\}$.
Then, measuring the degree of difficulty of witnessing the truth of $S$ is identified with that of finding a choice function for $f$.
Here, we consider choice problems for partial multivalued functions rather than relations in order to distinguish the hardest instance $f(x)=\emptyset$ and the easiest instance $x\in X\setminus{\rm dom}(f)$.

In this article, we only consider subspaces of $\N^\N$, so we can use the following simpler version of the Weihrauch reducibility.
For partial multivalued functions $f,g$, we say that $f$ is {\em Weihrauch reducible to $g$} (written $f\leq_{\sf W}g$) if there are partial computable functions $h,k$ such that $x\mapsto k(x,G\circ h(x))$ is a choice for $f$ whenever $G$ is a choice for $g$.
In other words,
\[(\forall x\in{\rm dom}(f))(\forall y)\;[y\in g(h(x))\implies k(x,y)\in f(x)].\]

In recent years, a lot of researchers has employed this notion to measure uniform computational strength of $\forall\exists$-theorems in analysis as an analogue of reverse mathematics.
Roughly speaking, the study of the Weihrauch lattice can be thought of as ``reverse mathematics plus uniformity minus proof theory.''
But this disregard for proof theory provides us a new insight into the classification of impredicative principles as we see in this article.
For more details on the Weihrauch lattice, we refer the reader to a recent survey article \cite{pauly-handbook}.

We use several operations on the Weihrauch lattice.
Given a partial multivalued function $f$, the {\em parallelization of $f$} is defined as follows:
\[\widehat{f}((x_n)_{n\in\om})=\prod_{n\in\om}f(x_n).\]

If $f\equiv_{\sf W}\hat{f}$, then we say that $f$ is {\em parallelizable}.
Given partial multivalued functions $f$ and $g$, the {\em compositional product of $f$ and $g$} (written $g\star f$) is a function which realizes the greatest Weihrauch degree among $g_0\circ f_0$ for $f_0\leq_{\sf W}f$ and $g_0\leq_{\sf W}g$.
It is known that such an operation $\star$ exists.
For basic properties of parallelization and compositional product, see also \cite{paulybrattka4}.

\section{Equivalence results in the Weihrauch lattice}

\subsection{$\Sigma^1_1$-Choice Principles}\label{section:Sigma-1-1-choice}

One of the main notions in this article is the $\Sigma^1_1$-choice principle.
In the context of the Weihrauch degrees, the $\Sigma^1_1$-choice principle on a space $X$ is formulated as the partial multivalued function which, given a code of a nonempty analytic set $A$, chooses an element of $A$.

We fix a coding system of all analytic sets in a Polish space $X$, and let $S_p$ be the analytic subset of $X$ coded by $p\in\om^\om$.
For instance, let $S_p$ be the projection of the $p$-th closed subset of $X\times \N^\N$ (i.e., the complement of the union of $p(n)$-th basic open balls) into the first coordinate (cf.\ \cite{KMP}).

The {\em $\Sigma^1_1$-choice principle on $X$}, $\Sigma^1_1$-${\sf C}_X$, is the partial multivalued function which, given a code of a nonempty analytic subset of $X$, chooses one element from $X$.
Formally speaking, it is defined as the following partial multivalued function:
\begin{align*}
{\rm dom}(\Sigma^1_1\mbox{-}{\sf C}_X)&=\{p\in\om^\om:S_p\not=\emptyset\},\\
\Sigma^1_1\mbox{-}{\sf C}_X(p)&=S_p.
\end{align*}

For basics on the $\Sigma^1_1$-choice principle on $X$, see also \cite{KMP}.
In a similar manner, one can also consider the $\Gamma$-choice principle on $X$, $\Gamma$-${\sf C}_X$, for any represented space $X$ and any collection $\Gamma$ of subsets of $X$ endowed with a representation $S_{\ast}\pcolon\om^\om\to\Gamma$.
We first describe how this choice principle is related to several very weak variants of the axiom of choice.

In logic, the {\em axiom of $\Gamma$ choice}, $\Gamma$-{\sf AC}, is known to be the following statement:
\[\forall a\exists b\;\varphi(a,b)\;\longrightarrow\;\exists f\forall a\;\varphi(a,f(a)),\]
where $\varphi$ is a $\Gamma$ formula.
If we require $a\in X$ and $b\in Y$, the above statement is written as $\Gamma$-${\sf AC}_{X\to Y}$.
We examine the complexity of a procedure that, given a $\Sigma^1_1$ formula $\varphi$ (with a parameter) satisfying the premise of $\Sigma^1_1$-${\sf AC}_{X\to Y}$, returns a choice for $\varphi$.
In other words, we interpret $\Sigma^1_1$-${\sf AC}_{X\to Y}$ as the following partial multivalued function:
\begin{align*}
{\rm dom}([\Sigma^1_1\mbox{-}{\sf AC}_{X\to Y}]_{\sf mv})&=\{p\in\om^\om:\forall a\exists b\;\langle a,b\rangle\in S_p\},\\
[\Sigma^1_1\mbox{-}{\sf AC}_{X\to Y}]_{\sf mv}(p)&=\{f\in Y^X:(\forall a)\;\langle a,f(a)\rangle\in S_p\}.
\end{align*}

Unfortunately, this interpretation is different from the usual (relative) realizability interpretation.
However, the above interpretation of $\Sigma^1_1$-${\sf AC}_{X\to \N}$ is related to a descriptive-set-theoretic notion known as the number uniformization property (or equivalently, the generalized reduction property) for $\Sigma^1_1$ (cf.\ \cite[Definition 22.14]{KechrisBook}).
In the context of Weihrauch degrees, the above interpretation is obviously related to the parallelization of the $\Sigma^1_1$-choice principle.

\begin{Obs}\label{obs:choice}
If $X$ is an initial segment of $\N$, then we have $\widehat{\Sigma^1_1\mbox{-}{\sf C}_X}\equiv_{\sf W}[\Sigma^1_1\mbox{-}{\sf AC}_{\N\to X}]_{\sf mv}$.
In particular, ${\sf C}_{\N^\N}\equiv_{\sf W}[\Sigma^1_1\mbox{-}{\sf AC}_{\N\to \N^\N}]_{\sf mv}$.
\qed
\end{Obs}

In logic, the {\em axiom of $\Sigma^1_1$-dependent choice on $X$} is the following statement:
\[\forall a\exists b\;\varphi(a,b)\;\longrightarrow\;\forall a\exists f\;[f(0)=a\;\&\;\forall n\;\varphi(f(n),f(n+1))],\]
where $\varphi$ is a $\Sigma^1_1$-formula, and $a$ and $b$ range over $X$.
Note that the dependent choice is equivalent to the statement saying that if $T$ is a definable pruned tree of height $\om$, then there is an infinite path through $T$.
However, this translation may change the logical complexity of a formula $\varphi$ and a tree $T$.
For this reason, we will use the symbol $\Sigma^1_1$-${\sf DC}_X$ to denote the scheme of the $\Sigma^1_1$-dependent choice on any analytic set $Y\subseteq X$ instead of considering a single space $X$.
Then we examine the complexity of a procedure that, given a $\Sigma^1_1$ set $Y\subseteq X$ and a $\Sigma^1_1$ formula $\varphi$ (with a parameter) satisfying the premise of the $\Sigma^1_1$-dependent choice on $Y$ and an element $a\in X$, returns $f$ satisfying the conclusion:
\begin{align*}
{\rm dom}([\Sigma^1_1\mbox{-}{\sf DC}_{X}]_{\sf mv})&=\{\langle p,q,a_0\rangle\in(\om^\om)^2\times S_q:\forall a\in S_q\exists b\in S_q\;\langle a,b\rangle\in S_p\},\\
[\Sigma^1_1\mbox{-}{\sf DC}_{X}]_{\sf mv}(p,q,a_0)&=\{f\in S_q^\N:f(0)=a_0\;\&\;\forall n\;\langle f(n),f(n+1)\rangle\in S_p\}.
\end{align*}

Note that this formulation is different from the $\Sigma^1_1$-dependent choice on $X$ in the context of second order arithmetic.
Indeed, our formulation falls between the $\Sigma^1_1$-dependent choice and the {\em strong $\Sigma^1_1$-dependent choice} (cf.\ Simpson \cite{SOSOA:Simpson}).
Now, it is easy to see the following:

\begin{Proposition}\label{prop:depecho}
${\sf C}_{\N^\N}\equiv_{\sf W}[\Sigma^1_1\mbox{-}{\sf DC}_{\N^\N}]_{\sf mv}\equiv_{\sf W}[\Sigma^1_1\mbox{-}{\sf DC}_\N]_{\sf mv}$.
\end{Proposition}

\begin{proof}
$[\Sigma^1_1\mbox{-}{\sf DC}_{\N^\N}]_{\sf mv}\leq_{\sf W}\Sigma^1_1\mbox{-}{\sf C}_{\N^\N}$:
The set of all solutions to an instance of $[\Sigma^1_1\mbox{-}{\sf DC}_{\N^\N}]_{\sf mv}$ is obviously $\Sigma^1_1$ relative to the given parameter, and one can easily find its $\Sigma^1_1$-index.

${\sf C}_{\N^\N}\leq_{\sf W}[\Sigma^1_1\mbox{-}{\sf DC}_\N]_{\sf mv}$:
Let $T$ be a pruned $\Sigma^1_1$ tree, and put $S_q=[T]$.
Then, let $\varphi_T(\sigma,\tau)$ be the formula expressing that $\tau$ is an immediate successor of $\sigma$.
Moreover, $\varphi_T$ satisfies the premise of $\Sigma^1_1\mbox{-}{\sf DC}_\N$ since $T$ is pruned.
Let $f$ be a solution to the instance $\varphi$ of $[\Sigma^1_1\mbox{-}{\sf DC}_\N]_{\sf mv}$ where $f(0)$ is the empty string.
Since $T$ is pruned, $f$ must be a path through $T$.

We conclude by remarking that $\Sigma^1_1\mbox{-}{\sf C}_{\N^\N}\leq_{\sf W}{\sf C}_{\N^\N}$: Given $p\in\om^\om$, one can find an element of $S_p$ by using ${\sf C}_{\N^\N}$ to find an element $x$ of the $p$-th closed set, and then taking the projection of $x$. Finally, obviously $[{\sf DC}_\N]_{\sf mv}\leq_{\sf W}[\Sigma^1_1\mbox{-}{\sf DC}_{\N^\N}]_{\sf mv}$.
\end{proof}

In the proper context, Question \ref{que:Borel-choice} was formulated as the problem asking whether $\widehat{\Sigma^1_1\mbox{-}{\sf C}_\N}<_{\sf W}{\sf C}_{\N^\N}$.
By the above observations, this is the same as asking the following.

\begin{Question}[Restatement of {Question \ref{que:Borel-choice}}]
Do we have $[{\Sigma^1_1}\mbox{-}{\sf AC}_{\N\to \N}]_{\sf mv}<_{\sf W}[\Sigma^1_1\mbox{-}{\sf DC}_{\N}]_{\sf mv}$?
Or equivalently, $[{\Sigma^1_1}\mbox{-}{\sf AC}_{\N\to \N}]_{\sf mv}<_{\sf W}[\Sigma^1_1\mbox{-}{\sf AC}_{\N\to\N^\N}]_{\sf mv}$?
\end{Question}

\subsection{Compact Choice Principles}\label{subsection:compact-choice}

According to the Arsenin-Kunugui uniformization theorem (cf.\ \cite[Theorem 18.10]{KechrisBook}), the choice principle for $\sigma$-compact $\Delta^1_1$ sets is much simpler than that for arbitrary ${\Delta}^1_1$ sets.
We are interested in that an analogous statement holds for $\Sigma^1_1$-choice, while we know that even a compact $\Sigma^1_1$-choice does not admit a Borel uniformization.

We now consider subprinciples of the $\Sigma^1_1$ choice principle by restricting its domain.
Recall that $S_p$ is the analytic set in $X$ coded by $p\in\om^\om$.
Let $\mathcal{R}$ be a collection of subsets of $X$.
Define the {\em $\Sigma^1_1$-choice principle $\Sigma^1_1\mbox{-}{\sf C}_X\rest{\mathcal{R}}$ restricted to sets in $\mathcal{R}$} as follows:
\begin{align*}
\Sigma^1_1\mbox{-}{\sf C}_{X}\rest{\mathcal{R}}&:\subseteq\om^\om\rightrightarrows X,\\
{\rm dom}(\Sigma^1_1\mbox{-}{\sf C}_{X}\rest{\mathcal{R}})&=\{p\in\om^\om:S_p\not=\emptyset\mbox{ and }S_p\in\mathcal{R}\},\\
\Sigma^1_1\mbox{-}{\sf C}_{X}\rest{\mathcal{R}}(p)&=S_p
\end{align*}

First, we consider the $\Sigma^1_1$ choice principle restricted to compact sets, that is, we define the {\em compact $\Sigma^1_1$-choice $\Sigma^1_1\mbox{-}{\sf KC}_X$} as follows:
\[\Sigma^1_1\mbox{-}{\sf KC}_X=\Sigma^1_1\mbox{-}{\sf C}_X\rest{\{A\subseteq X:A\text{ is compact}\}.}\]

In other words, the $\Sigma^1_1$-compact choice principle $\Sigma^1_1\mbox{-}{\sf KC}_X$ is the multivalued function which, given a code of a nonempty compact $\Sigma^1_1$ set, chooses one element from the set.
This choice principle can be thought of as an interpretation of parallelized two-valued choice.
Before confirming the equivalence, first note that in \cite{KMP} the $\Sigma^1_1$ parallelized two-valued choice is shown to be equivalent to the following principles:

\begin{itemize}
\item The principle $\Sigma^1_1\mbox{-}{\sf WKL}$, the {\em weak K\"onig's lemma for $\Sigma^1_1$-trees}, is the partial multivalued function which, given a binary tree $T\subseteq 2^{<\om}$ which is $\Sigma^1_1$ relative to a given parameter, chooses an infinite path through $T$.
\item The principle $\Pi^1_1\mbox{-}{\sf Sep}$, the {\em problem of separating a disjoint pair of $\Pi^1_1$ sets}, is the partial multivalued function which, given a pair of disjoint sets $A,B\subseteq\om$ which are $\Pi^1_1$ relative to a given parameter, chooses a set $C\subseteq\om$ separating $A$ from $B$, that is, $A\subseteq C$ and $B\cap C=\emptyset$.
\end{itemize}

\begin{Fact}[Kihara-Marcone-Pauly \cite{KMP}]\label{KMP}
$\widehat{\Sigma^1_1\mbox{-}{\sf C}_2}\equiv_{\sf W}\Pi^1_1\mbox{-}{\sf Sep}\equiv_{\sf W}\Sigma^1_1\mbox{-}{\sf WKL}$.
\end{Fact}

We now show that these are equivalent to the $\Sigma^1_1$-compact choice.

\begin{Proposition}\label{prop:compact-prouduct-of-two}
$\widehat{\Sigma^1_1\mbox{-}{\sf C}_2}\equiv_{\sf W}\Sigma^1_1\mbox{-}{\sf KC}_{\N^\N}\equiv_{\sf W}[\Sigma^1_1\mbox{-}{\sf AC}_{\N\to 2}]_{\sf mv}$.
\end{Proposition}
\begin{proof}
By Observation \ref{obs:choice}, we have $\widehat{\Sigma^1_1\mbox{-}{\sf C}_2}\equiv_{\sf W}[\Sigma^1_1\mbox{-}{\sf AC}_{\N\to 2}]_{\sf mv}$.
  To show that these are equivalent to the $\Sigma^1_1$ compact choice principle, we claim that a set $A\subseteq\Baire$ is $\Sigma^1_1$ and compact if and only if it is computably isomorphic to a $\Sigma^1_1$-closed set $B\subseteq\Cantor$.
  The reverse implication is clear, as compactness is preserved via continuous functions. So suppose that $A$ is $\Sigma^1_1$ and compact. First, it is clearly closed, so let $T_b\subseteq\baire$ be a $\Sigma^1_1$ tree such that $A=[T_b]$ and $T_b$ has no dead-end. For every $\sigma$, there exists at most finitely many $i\in\Nb$ such that $\sigma\fr i\in T_b$, and this fact is observed at some stage $\alpha_\sigma$ below $\omCK$.
  Now apply $\Sigma^1_1$-boundedness to the total function $\sigma\mapsto \alpha_\sigma$ to get a stage $\alpha$ below $\omCK$ such that already, $T_b[\alpha]$ is a finitely branching tree. Then, we can use the usual injection of a finitely branching tree space into Cantor space.
By uniformly relativizing this argument, we now obtain $\Sigma^1_1\mbox{-}{\sf WKL}\equiv_{\sf W}\Sigma^1_1\mbox{-}{\sf KC}_{\N^\N}$, which can conclude by invoking Fact~\ref{KMP} that assert $\Sigma^1_1\mbox{-}{\sf WKL}\equiv_{\sf W}\widehat{\Sigma^1_1\mbox{-}{\sf C}_2}$.
\end{proof}

Next, we show that the compact $\Sigma^1_1$-choice principle is also Weihrauch equivalent to the following principles:
\begin{itemize}
\item The principle $\Pi^1_1\mbox{-}{\sf Tot}_2$, the {\em totalization problem for partial $\Pi^1_1$ two-valued functions}, is the partial multivalued function which, given a partial function $\varphi\pcolon\om\to 2$ which is $\Pi^1_1$ relative to a given parameter, chooses a total extension $f\colon\om\to 2$ of $\varphi$.
\item The principle $\Pi^1_1\mbox{-}{\sf DNC}_2$, the {\em problem of finding a two-valued diagonally non-$\Pi^1_1$ function}, is the partial multivalued function which, given a sequence of partial functions $(\varphi_e)_{e\in\om}$ which are $\Pi^1_1$ relative to a given parameter, chooses a total function $f\colon\om\to 2$ diagonalizing the sequence, that is, $f(e)\not=\varphi_e(e)$ whenever $\varphi_e(e)$ is defined.
\end{itemize}

The latter notion has also been studied by Kihara-Marcone-Pauly \cite{KMP}.

\begin{Proposition}\label{prop:compact-choice-total-DNC}
$\widehat{\Sigma^1_1\mbox{-}{\sf C}_2}\equiv_{\sf W}\Pi^1_1\mbox{-}{\sf Tot}_2\equiv_{\sf W}\Pi^1_1\mbox{-}{\sf DNC}_2$.
\end{Proposition}

\begin{proof}
$\Pi^1_1\mbox{-}{\sf Tot}_2\leq_{\sf W}\Pi^1_1\mbox{-}{\sf DNC}_2$:
Given a partial function $\varphi\pcolon\om\to 2$, define $\psi_e(e)=1-\varphi(e)$.
If $g$ diagonalizes $(\psi_e)_{e\in\om}$, then $g(e)=1-\psi_e(e)=\varphi(e)$ whenever $\varphi(e)$ is defined.
Therefore, $g$ is a totalization of $\varphi$.

$\Pi^1_1\mbox{-}{\sf DNC}_2\leq_{\sf W}\widehat{\Sigma^1_1\mbox{-}{\sf C}_2}$: Define $S_e=\{a:\varphi_e(e)\downarrow<2\;\rightarrow\;a\not=\varphi_e(e)\}$ is uniformly $\Sigma^1_1$.
Moreover, the choice for $(S^e_{n})_{n\in\om}$ clearly diagonalizes $(\varphi_e)_{e\in\om}$.

$\widehat{\Sigma^1_1\mbox{-}{\sf C}_2}\leq_{\sf W}\Pi^1_1\mbox{-}{\sf Tot}_2$:
Given a $\Sigma^1_1$ set $S_n\subseteq 2$, wait for $S_n$ becomes a singleton, say $S_n=\{s_n\}$.
It is easy to find an index of a partial $\Pi^1_1$ function $f$ such that $f(n)=s_n$ whenever $S_n=\{s_n\}$.
Then, any total extension of $f$ is a choice for $(S_n)_{n\in\om}$.
\end{proof}

A set is {\em $\sigma$-compact} if it is a countable union of compact sets.
By Saint Raymond's theorem (cf.\ \cite[Theorem 35.46]{KechrisBook}), any Borel set with $\sigma$-compact sections can be written as a countable union of Borel sets with compact sections.
In particular, a Borel code for a $\sigma$-compact set $S$ can be transformed into a uniform sequence of Borel codes of compact sets whose union is $S$.
However, there is no analogous result for analytic sets (cf.\ Steel \cite{Ste80}).
Therefore, we do not introduce the $\sigma$-compact $\Sigma^1_1$-choice as 
\[\Sigma^1_1\mbox{-}{\sf C}_X\rest{\{A\subseteq X:A\text{ is $\sigma$-compact}\}.}\]

Instead, we directly code an analytic $\sigma$-compact set as a sequence of analytic codes of compact sets.
In other words, the {\em $\sigma$-compact $\Sigma^1_1$-choice principle}, $\Gamma\mbox{-}{\sf K}_\sigma{\sf C}_{X}$, is the partial multivalued function which, given a sequence $(S_n)_{n\in\N}$ of compact $\Sigma^1_1$ (relative to a parameter) sets at least one of which is nonempty, chooses an element from $\bigcup_{n\in\N}S_n$.
Equivalently (modulo the Weihrauch equivalence), one can formalize $\Sigma^1_1\mbox{-}{\sf K}_\sigma{\sf C}_{\N^\N}$ as the compositional product $\Sigma^1_1\mbox{-}{\sf KC}_{\N^\N}\star\Sigma^1_1\mbox{-}{\sf C}_\N$.

\subsection{Restricted Choice Principles}\label{sec:restricted-choice-principles}

Next, we consider several variations of the axiom of choice:
\begin{enumerate}
\item The axiom of {\em unique choice}: $\forall a\exists! b\;\varphi(a,b)\;\longrightarrow\;\exists f\forall a\;\varphi(a,f(a))$.
\item The axiom of {\em finite choice}: For any $a$, if $\{b:\varphi(a,b)\}$ is nonempty and finite, then there is a choice function for $\varphi$, that is, $\exists f\forall a\;\varphi(a,f(a))$.
\item The axiom of {\em cofinite choice}: For any $a$, if $\{b:\varphi(a,b)\}$ is cofinite, then there is a choice function for $\varphi$.
\item The axiom of {\em finite-or-cofinite choice}: For any $a$, if $\{b:\varphi(a,b)\}$ is either nonempty and finite or cofinite, then there is a choice function for $\varphi$.
\item The axiom of {\em total unique choice}: $\exists f\forall a\;[\exists!b\varphi(a,b)\;\longrightarrow\;\varphi(a,f(a))]$.
\end{enumerate}

The last notion is a modification of a variant of hyperarithmetical axiom of choice introduced by Tanaka \cite{Tan03} in the context of second order arithmetic, where the original formulation is given as follows:
\[\exists Z\forall n\;[\exists !X\varphi(n,X)\ \longrightarrow\ \varphi(n,Z_n)],\]
where $\varphi$ is a $\Sigma^1_1$ formula.
We interpret these axioms of choice as parallelization of partial multi-valued functions.
Then, we define:
\begin{align*}
\Sigma^1_1\mbox{-}{\sf UC}_{X}&=\Sigma^1_1\mbox{-}{\sf C}_{X}\rest{\{A\subseteq X:|A|=1\}},\\
\Sigma^1_1\mbox{-}{\sf C}^{\sf fin}_{X}&=\Sigma^1_1\mbox{-}{\sf C}_{X}\rest{\{A\subseteq X:A\text{ is finite}\}},\\
\Sigma^1_1\mbox{-}{\sf C}^{\sf cof}_{X}&=\Sigma^1_1\mbox{-}{\sf C}_{X}\rest{\{A\subseteq X:A\text{ is cofinite}\}},\\
\Sigma^1_1\mbox{-}{\sf C}^{\sf foc}_{X}&=\Sigma^1_1\mbox{-}{\sf C}_{X}\rest{\{A\subseteq X:A\text{ is finite or cofinite}\}},\\
{\Sigma^1_1\mbox{-}{\sf C}^{\sf aof}_{X}}&{=\Sigma^1_1\mbox{-}{\sf C}_{X}\rest{\{A\subseteq X:A=X\text{ or $A$ is finite}\}},}\\
\Sigma^1_1\mbox{-}{\sf C}^{\sf aou}_{X}&=\Sigma^1_1\mbox{-}{\sf C}_{X}\rest{\{A\subseteq X:A= X\text{ or }|A|=1\}}.
\end{align*}

Note that the all-or-unique choice is often denoted by ${\sf AoUC}_X$ instead of ${\sf C}^{\sf aou}_X$, cf.\ \cite{pauly-kihara2-mfcs}.
Among others, we see that the all-or-unique choice $\Sigma^1_1\mbox{-}{\sf C}^{\sf aou}_{\N}$ is quite robust.
Recall from Proposition \ref{prop:compact-choice-total-DNC}
that the $\Pi^1_1$-totalization principle $\Pi^1_1\mbox{-}{\sf Tot}_2$ and the $\Pi^1_1$-diagonalization principle $\Pi^1_1\mbox{-}{\sf DNC}_2$ restricted to {\em two valued functions} are equivalent to the $\Sigma^1_1$ compact choice principle.
We now consider the $\om$-valued versions of the totalization and the diagonalization principles:
\begin{itemize}
\item The principle $\Pi^1_1\mbox{-}{\sf Tot}_\N$, the {\em totalization problem for partial $\Pi^1_1$ functions}, is the partial multivalued function which, given a partial function $\varphi\pcolon\om\to\om$ which is $\Pi^1_1$ relative to a given parameter, chooses a total extension of $\varphi$.
\item The principle $\Pi^1_1\mbox{-}{\sf DNC}_\N$, the {\em problem of finding a diagonally non-$\Pi^1_1$ function}, is the partial multivalued function which, given a sequence of partial functions $(\varphi_e)_{e\in\om}$ which are $\Pi^1_1$ relative to a given parameter, chooses a total function $f:\om\to\om$ diagonalizing the sequence.
\end{itemize}

It is clear that $\Pi^1_1\mbox{-}{\sf DNC}_\N\leq_{\sf W}\Pi^1_1\mbox{-}{\sf DNC}_2\equiv_{\sf W}\Pi^1_1\mbox{-}{\sf Tot}_2\leq_W\Pi^1_1\mbox{-}{\sf Tot}_\N$.
One can easily see the following.

\begin{Proposition}\label{prop:aouc-total-n}
$\widehat{\Sigma^1_1\mbox{-}{\sf C}^{\sf aou}_\N}\equiv_{\sf W}\Pi^1_1\mbox{-}{\sf Tot}_\N$.
\end{Proposition}

\begin{proof}
The argument is almost the same as Proposition \ref{prop:compact-choice-total-DNC}.
Given a partial function $\varphi$, define $S_{n}=\{a:\varphi(n)\downarrow\;\rightarrow\;a=\varphi(n)\}$, which is uniformly $\Sigma^1_1$.
Clearly, either $S_n=\N$ or $S_n$ is a singleton.
Hence, the all-or-unique choice principle chooses an element of $S_{n}$, which produces a totalization of $\varphi$.

Conversely, given a $\Sigma^1_1$ set $S_n\subseteq\N$, wait until $S_n$ becomes a singleton, say $S_n=\{s_n\}$.
It is easy to find an index of partial $\Pi^1_1$ function $f$ such that $f(n)=s_n$ whenever $S_n=\{s_n\}$.
Then, any total extension of $f$ is a choice for $(S_n)_{n\in\om}$.
\end{proof}

We introduce the {\em totalization of the $\Sigma^1_1$-choice principle (restricted to $\mathcal{R}$) on $X$}.
Recall that $S_p$ is the analytic set in $X$ coded by $p\in\om^\om$.
Then we define $\Sigma^1_1\mbox{-}{\sf C}^{\sf tot}_X\rest{\mathcal{R}}$ as follows:
\begin{align*}
\Sigma^1_1\mbox{-}{\sf C}^{\sf tot}_{X}\rest{\mathcal{R}}&:\om^\om\rightrightarrows \N,\\
{\rm dom}(\Sigma^1_1\mbox{-}{\sf C}^{\sf tot}_{X}\rest{\mathcal{R}})&=\{p\in\om^\om:S_p\not=\emptyset\mbox{ and }S_p\in\mathcal{R}\},\\
\Sigma^1_1\mbox{-}{\sf C}^{\sf tot}_{X}\rest{\mathcal{R}}(p)&=
\begin{cases}
S_p&\mbox{ if $x\in\mathcal{R}$,}\\
X&\mbox{ otherwise.}
\end{cases}
\end{align*}

Roughly speaking, if a given $\Sigma^1_1$ set $S$ is nonempty and belongs to $\mathcal{R}$, then any element of $S$ is a solution to this problem as a usual choice problem, but even if a set $S$ is either empty or does not belong to $\mathcal{R}$, there is a need to feed some value, although any value is acceptable as a solution.

In second order arithmetic, the totalization of dependent choice is known as {\em strong dependent choice} (cf.\ Simpson \cite[Definition VII.6.1]{SOSOA:Simpson}).
In the Weihrauch context, Kihara-Marcone-Pauly \cite{KMP} have found that the totalization of ${\sf C}_{\N^\N}$ has an important role in the study of the Weihrauch counterpart of arithmetical transfinite recursion.
Here we consider the totalization of $\Sigma^1_1\mbox{-}{\sf UC}_{\N^\N}$, which can be viewed as the multivalued version of the axiom of {\em total unique choice} mentioned above.

\begin{Proposition}
Let $X$ be a $\Delta^1_1$ subset of $\N$.
Then, $\Sigma^1_1\mbox{-}{\sf UC}^{\sf tot}_{X}\equiv_{\sf W}\Sigma^1_1\mbox{-}{\sf C}^{\sf aou}_{X}$.
\end{Proposition}

\begin{proof}
$\Sigma^1_1\mbox{-}{\sf UC}^{\sf tot}_{X}\leq_{\sf W}\Sigma^1_1\mbox{-}{\sf C}^{\sf aou}_{X}$:
Given a $\Sigma^1_1$ set $S$, wait until $S$ becomes a singleton at some ordinal stage.
If it happens, let $R=S$; otherwise keep $R=X$.
One can effectively find a $\Sigma^1_1$-index of $R$, and either $R=X$ or $R$ is a singleton.

$\Sigma^1_1\mbox{-}{\sf C}^{\sf aou}_{X}\leq_{\sf W}\Sigma^1_1\mbox{-}{\sf UC}^{\sf tot}_{X}$:
Trivial.
\end{proof}

In particular, the totalization of two-valued unique choice is equivalent to the compact choice.

\begin{Corollary}\label{cor:tvunch-comch}
$\widehat{\Sigma^1_1\mbox{-}{\sf UC}^{\sf tot}_{\sf 2}}\equiv_{\sf W}\Sigma^1_1\mbox{-}{\sf KC}_{\N^\N}$.
\end{Corollary}

\begin{proof}
It is clear that $\Sigma^1_1\mbox{-}{\sf C}^{\sf aou}_{\sf 2}\equiv_{\sf W}\Sigma^1_1\mbox{-}{\sf C}_{\sf 2}$.
Thus, the assertion follows from Fact \ref{KMP} and Proposition \ref{prop:compact-prouduct-of-two}.
\end{proof}

\subsection{Arithmetical Transfinite Recursion}

In reverse mathematics, the axiom of $\Sigma^1_1$-choice $\Sigma^1_1\mbox{-}{\sf AC}_0$ is known to be weaker than the arithmetical transfinite recursion scheme ${\sf ATR}_0$ (cf.\ \cite[Section VIII.4]{SOSOA:Simpson}).
However, an analogous result does not hold in the Weihrauch context.
The purpose of this section is to clarify the relationship between the $\Sigma^1_1$-choice principles and the arithmetical transfinite recursion principle in the Weihrauch lattice.

Kihara-Marcone-Pauly \cite{KMP} first introduced an analogue of {\em arithmetical transfinite recursion}, ${\sf ATR}_0$, in the context of Weihrauch degrees, and studied two-sided versions of several dichotomy theorems related to ${\sf ATR}_0$, but they have only considered the one-sided version of ${\sf ATR}_0$.
Then, Goh \cite{Goh} introduced the two-sided version of ${\sf ATR}_0$ to examine the Weihrauch strength of K\"onig's duality theorem for infinite bipartite graphs.
Roughly speaking, the above two Weihrauch problems are introduced as follows:
\begin{itemize}
\item The {\em one-sided version}, ${\sf ATR}$, by \cite{KMP} is the partial multivalued function which, given a countable well-ordering $\prec$, returns the jump hierarchy for $\prec$.
\item The {\em two-sided version}, ${\sf ATR}_2$, by \cite{Goh} is the total multivalued function which, given a countable linear ordering $\prec$, chooses either a jump hierarchy for $\prec$ or an infinite $\prec$-decreasing sequence.
\end{itemize}

Here, a {\em jump hierarchy} for a partially ordered set $(P,<_P)$ is a sequence $(H_p)_{p\in P}$ of sets satisfying the following property:
For all $p\in P$,
\[H_p=\bigoplus_{q<_Pp}H_q'.\]

Even if $\prec$ is not well-founded, some solution to ${\sf ATR}_{2}(\prec)$ may produce a jump hierarchy for $\prec$ (often called a {\em pseudo-hierarchy}) by Harrison's well-known result that there is a pseudo-well-order which admits a jump hierarchy (but a jump hierarchy is not necessarily unique).
Regarding ${\sf ATR}_{2}$, we note that, sometimes in practice, what we need is not a full jump hierarchy for a pseudo-well-ordering, but a jump hierarchy for an initial segment of $\prec$ containing its well-founded part.
Therefore, we introduce another two-sided version ${\sf ATR}_{2'}$ as follows:

Let $L$ be a linearly ordered set.
The {\em well-founded part} of $L$ is the largest initial segment of $L$ which is well-founded.
We say that an initial segment $I$ of $L$ is {\em large} if it contains a well-founded part of $L$.

We consider a variant of the arithmetical transfinite recursion ${\sf ATR}_{2'}$, which states that for any $x$-th linear order $\prec_x$, one can find either a jump hierarchy for a large initial segment of $\prec_x$ or an infinite $\prec_x$-decreasing sequence:
\begin{align*}
{\sf ATR}_{2'}(x)=\ &\{0\fr H:H\mbox{ is a jump hierarchy for a large initial segment of $\prec_x$}\}\\
\cup\ &\{1\fr p:p\mbox{ is an infinite decreasing sequence with respect to $\prec_x$}\}.
\end{align*}

Seemingly, ${\sf ATR}_{2'}$ is completely unrelated to any other choice principles.
Surprisingly, however, we will see that (the parallelization of) ${\sf ATR}_{2'}$ is arithmetically equivalent to the choice principle for $\Sigma^1_1$-compact sets, which is also equivalent to the $\Pi^1_1$ separation principle.
We say that $f$ is {\em arithmetically Weihrauch reducible to $g$} (written $f\leq_{\sf W}^ag$) if we are allowed to use arithmetic functions $H$ and $K$ (i.e., $H,K\leq_{\sf W}\lim^{(n)}$ for some $n\in\N$) in the definition of Weihrauch reducibility.

\begin{Theorem}\label{thm:ATR-is-equivalent-to-compact-choice}
$\widehat{{\sf ATR}_{2'}}\equiv_{\sf W}^a\Sigma^1_1\mbox{-}{\sf KC}_{\N^\N}\equiv_{\sf W}^a\widehat{\Sigma^1_1\mbox{-}{\sf C}^{\sf aou}_\N}$.
\end{Theorem}

We divide the proof of Theorem \ref{thm:ATR-is-equivalent-to-compact-choice} into two lemmas.

\begin{Lemma}\label{lem:ATRcc1}
${\sf ATR}_{2'}\leq_{\sf W}^a\widehat{\Sigma^1_1\mbox{-}{\sf UC}^{\sf tot}_2}$.
\end{Lemma}

\begin{proof}
Fix $x$.
Given $n\in\N$, let $JH_{n}$ be the set of jump hierarchies for $\prec_x\rest n$.
Note that $JH$ is an arithmetical relation.
For $a,k\in\N$, if $a\prec_x n$ then let $S^n_{a,k}$ be the set of all $i<2$ such that for some jump hierarchy $H\in JH_n$, the $k$-th value of the $a$-th rank of $H$ is $i$, that is, $H_a(k)=i$.
Otherwise, let $S^n_{a,k}=\{0\}$.
Clearly, $S^n_{a,k}$ is $\Sigma^1_1$ uniformly in $n,a,k$, and therefore there is a computable function $f$ such that $S^n_{a,k}$ is the $f(n,a,k)$-th $\Sigma^1_1$ set $G_{f(n)}$.
Note that if $\prec_x\rest n$ is well-founded, then the product $\prod_{\langle a,k\rangle}S^n_{a,k}$ consists of a unique jump hierarchy for $\prec_x\rest n$.
In particular, $S^n_{a,k}$ is a singleton for any $a\prec_xn$ and $k\in\N$ whenever $\prec_x\rest n$ is well-founded.

Given $p_{n,a,k}\in \Sigma^1_1\mbox{-}{\sf UC}^{\sf tot}_{2}(f(n,a,k))$, define $H_n=\bigoplus_{\langle a,k\rangle}p_{n,a,k}$.
Note that if $n$ is contained in the well-founded part of $\prec_x$, then $H_n$ must be a jump hierarchy for $\prec_x\rest n$.
By using an arithmetical power, first ask if $H_n$ is a jump hierarchy for $\prec_{x}\rest n$ for every $n$.
If yes, $\bigoplus_nH_n$ is a jump hierarchy along the whole ordering $\prec_x$, which is, in particular, large.
If no, next ask if there exists a $\prec_{x}$-least $n$ such $H_n$ is not a jump hierarchy for $\prec_{x}\rest n$.
If yes, choose such an $n$, and then obviously $n$ is not contained in the well-founded part of $\prec_x$.
Hence, $\prec_x\rest n$ is a large initial segment of $\prec_x$.
Moreover, by minimality of $n$, $\bigoplus\{H_j':j\prec_x n\}$ is the jump hierarchy for $\prec_x\rest n$.
If there is no such $n$, let $j_0$ be the $<_\N$-least number such that $H_{j_0}$ is not a jump hierarchy for $\prec_x\rest {j_0}$, and $j_{n+1}\prec_{x}j_n$ be the $<_\N$-least number such that $H_{j_{n+1}}$ is not a jump hierarchy for $\prec_x\rest {j_{n+1}}$.
By using an arithmetical power, one can find such an infinite sequence $(j_n)_{n\in\om}$, which is clearly decreasing with respect to $\prec_x$.
\end{proof}

\begin{Lemma}\label{lem:AoUC-reducible-to-ATR}
$\Sigma^1_1\mbox{-}{\sf C}^{\sf aou}_\N\leq_{\sf W}^a\widehat{{\sf ATR}_{2'}}$.
\end{Lemma}

\begin{proof}
Let $S$ be a computable instance of $\Sigma^1_1\mbox{-}{\sf C}^{\sf aou}_\N$.
Let $\prec_n$ be a linear order on an initial segment $L_n$ of $\N$ such that $n\in S$ iff $\prec_n$ is ill-founded.
Let $H_n$ be a solution to the instance $\prec_n$ of ${\sf ATR}_{2'}$.
Ask if there is $n$ such that $H_n$ is an infinite decreasing sequence w.r.t.\ $\prec_n$.
If so, one can arithmetically find such an $n$, which belongs to $S$.
Otherwise, each $H_n$ is a jump hierarchy along a large initial segment $J_n$ of $L_n$.
By an arithmetical way, one can obtain $J_n$.
Then ask if $L_n\setminus J_n$ is nonempty, and has no $\prec_n$-minimal element.
If the answer to this arithmetical question is yes, we have $n\in S$.

Thus, we assume that for any $n$ either $L_n=J_n$ holds or $L_n\setminus J_n$ has a $\prec_n$-minimal element.
In this case, if $n\in S$ then $J_n$ is ill-founded.
This is because if $J_n$ is well-founded, then $J_n$ is exactly the well-founded part of $L_n$ since $J_n$ is large, and thus $L_n\setminus J_n$ is nonempty and has no $\prec_n$-minimal element.
Moreover, since $J_n$ admits a jump hierarchy while it is ill-founded, $J_n$ is a pseudo-well-order; hence $H_n$ computes all hyperarithmetical reals.
Conversely, if $n\not\in S$ then $H_n$ is a jump hierarchy along the well-order $J_n=L_n$, which is hyperarithmetic.

Now, ask if the following $(H_n)_{n\in\N}$-arithmetical condition holds:
\begin{align}\label{equ:compare-jump-hierarchies}
(\exists i)(\forall j)\;H_i\not<_TH_j.
\end{align}

By our assumption that $S\not=\emptyset$, there is $j\in S$, so that $H_j$ computes all hyperarithmetic reals.
Therefore, if (\ref{equ:compare-jump-hierarchies}) is true, for such an $i$, the hierarchy $H_i$ cannot be hyperarithmetic; hence $i\in S$.
Then one can arithmetically find such an $i$.
If (\ref{equ:compare-jump-hierarchies}) is false, for any $i$ there is $j$ such that $H_i<_TH_j$.
This means that there are infinitely many $i$ such that $H_i$ is not hyperarithmetical, i.e., $i\in S$.
However, by our assumption, if $S$ is infinite, then $S=\N$.
Hence, any $i$ is solution to $S$.

Finally, one can uniformly relativize this argument to any instance $S$.
\end{proof}

\begin{proof}[Proof of Theorem \ref{thm:ATR-is-equivalent-to-compact-choice}]
By Corollary \ref{cor:tvunch-comch} and Lemmas \ref{lem:ATRcc1} and \ref{lem:AoUC-reducible-to-ATR}.
\end{proof}

One can also consider a jump hierarchy for a partial ordering.
Then, we consider the following partial order version of Goh's arithmetical transfinite recursion.
Let $(\prec_x)$ be a list of all countable partial orderings.
\begin{align*}
{\sf ATR}_{2}^{\sf po}(x)=\ &\{0\fr H:H\mbox{ is a jump hierarchy for $\prec_x$}\}\\
\cup\ &\{1\fr p:p\mbox{ is an infinite decreasing sequence with respect to $\prec_x$}\}.
\end{align*}

Note that ${\sf ATR}_{2}^{\sf po}(x)$ is an arithmetical subset of $\N^\N$.
Obviously, 
\[{\sf ATR}\leq_{\sf W}{\sf ATR}_{2'}\leq_{\sf W}{\sf ATR}_2\leq_{\sf W}{\sf ATR}_{2}^{\sf po}\leq_{\sf W}\Sigma^1_1\mbox{-}{\sf C}_{\N^\N}.\]
This version of arithmetical transfinite recursion directly computes a solution to the all-or-unique choice on the natural numbers without using parallelization or arithmetical power.

\begin{Proposition}
$\Sigma^1_1\mbox{-}{\sf C}^{\sf aou}_\N\leq_{\sf W}{\sf ATR}_{2}^{\sf po}$.
\end{Proposition}

\begin{proof}
Let $S$ be a computable instance of $\Sigma^1_1\mbox{-}{\sf C}^{\sf aou}_\N$.
Let $T_n$ be a computable tree such that $n\in S$ iff $T_n$ is ill-founded.
Define 
\[T=00\sqcup_nT_n=\{\langle\rangle,\langle 0\rangle,\langle 00\rangle\}\cup\{\langle 00n\rangle\sigma:\sigma\in T_n\}.\]
Let $i\fr H$ be a solution to the instance $T$ of ${\sf ATR}_{2}^{\sf po}$.
If $i=1$, i.e., if $H$ is an infinite decreasing sequence w.r.t.\ $T$, then this provides an infinite path $p$ through $T$.
Then, choose $n$ such that $00n\prec p$, which implies $T_n$ is ill-founded, and thus $n\in S$.
Otherwise, $i=0$, and thus $H$ is a jump hierarchy for $T$.
We define $H^\ast_n=H_{\langle 00n\rangle}$.
Note that if $n\not\in S$ then $H^\ast_n$ is hyperarithmetic, and if $n\in S$ then $H^\ast_n$ computes all hyperarithmetical reals.
By the definition of a jump hierarchy, we have $(H^\ast_n)''\leq_TH$.
Thus, the following is an $H$-computable question:
\begin{align}\label{equ:compare-jump-hierarchies2}
(\exists i)(\forall j)\;H^\ast_i\not<_TH^\ast_j.
\end{align}

As in the proof of Lemma \ref{lem:AoUC-reducible-to-ATR}, one can show that if (\ref{equ:compare-jump-hierarchies2}) is true for $i$ then $i\in S$, and if (\ref{equ:compare-jump-hierarchies2}) is false then any $i$ is a solution to $S$.
As before, one can uniformly relativize this argument to any instance $S$.
\end{proof}

\begin{Question}
${\sf ATR}_{2}\equiv^a_{\sf W}{\sf ATR}_{2'}\equiv^a_{\sf W}{\sf ATR}_{2}^{\sf po}$?
\end{Question}

\section{The Medvedev lattice of $\Sigma^1_1$-Closed Sets}

In this section, we investigate the structure of different (semi-)sublattices of the Medvedev degrees, corresponding to restrictions on the axiom of choice. The Medvedev reduction was introduced in \cite{Medvedev} to classify problems according to their degree of difficulty, as for Weihrauch reducibility. However, when Weihrauch reducibility compare problems that have several instances, each of them with multiple solutions, Medvedev reducibility compare ``mass problems'', which correspond to problems with a unique instance. A mass problem is a set of functions from natural numbers to natural numbers, representing the set of solutions. For two mass problems $P,Q\subseteq\omega^\omega$, we say that $P$ is Medvedev reducible to $Q$ if every solution for $Q$ uniformly computes a solution for $P$.

\begin{Definition}[Medvedev reduction]
  Let $P,Q\subseteq\omega^\omega$ be sets. We say that $P$ is Medvedev reducible to $Q$, written $P\leq_MQ$ if there exists a single computable function $f$ such that for every $x\in B$, $f(x)\in A$.
\end{Definition}

If $P\subseteq X\times\omega^\omega$ is now a Weihrauch problem, that is a partial multi-valued function, then for any instance $x\in X$, one can consider the mass problem $P(x)=\{y:(x,y)\in P\}$. Using Medvedev reducibility, we are able to compare the degree of complexity of different instances of the same problem, and we will be interested in the structural property of their complexity: Given a Weihrauch problem $P$, we define the Medvedev lattice of $P$ by the lattice of Medvedev degrees of $P(x)$ for all computable instances $x\in\dom(P)$.

We will be mainly interested in upward density of Medvedev lattices of $P$, for $P$ being various choice problems, as it can be used to Weihrauch-separate two problems. Suppose that $P\leq_WQ$ and the Medvedev lattice of $Q$ is upward dense while the Medvedev lattice of $P$ is not. Then, we have $P<_WQ$: Let $x\in\dom(P)$ be any computable instance realizing a maximal $P$-Medvedev degree, and take $y\in\dom(Q)$ such that $P(x)\leq_MQ(y)$ (as $P\leq_WQ$). By upward density, let $z\in\dom(Q)$ be such that $Q(z)>_MQ(y)$. Then, it cannot be that there is $t\in\dom(P)$ such that $P(t)\geq_MQ(z)$, as it would contradict maximality of $x$. Therefore, $z$ is a witness that $P<_WQ$. 

We will consider several restricted $\Sigma^1_1$ closed subsets of Baire space, defined as below.

\begin{Definition}
  We define several versions of axiom of choice where the set we have to choose from are restricted to special kinds:
  \begin{align*}
    \Sigma^1_1\mbox{-}{\sf AC}^{\star}_{\N}&=\widehat{\Sigma^1_1\mbox{-}{\sf C}^{\star}_{\N}}
  \end{align*}
  where $\star\in\{\sf fin, cof, foc, aof, aou\}$ respectively corresponding to ``finite'', ``cofinite'', ``finite or cofinite'', ``all or finite'' and  ``all or unique''. Note that we drop the multivalued notation $[\cdot]_{\sf mv}$.
  We will also consider the Dependent Choice with the same restricted sets:
  \begin{align*}
    \Sigma^1_1\mbox{-}{\sf DC}^\star_{\N}&=\Sigma^1_1\mbox{-}{\sf C}_{\N^\N}\rest{\{[T]: \forall\sigma\in T, \{n:\sigma\fr n\in T\}\text{ is }\star\}}.
  \end{align*}
  where $\star\in\{\sf fin, cof, foc, aof, aou\}$ has the same meaning. For any $\sigma\in\omega^{<\omega}$ a string corresponding to a choice for the previous sets, $\{n:\sigma\fr n\in T\}$ corresponds to the next possible choice, and this set has to be as specified by $\star$. Note that it corresponds to a particular formulation of $\Sigma^1_1$ dependent choice, as explained just before Proposition~\ref{prop:depecho}.
\end{Definition}

Throughout this section, we use the following abuse of notation.

\begin{Notation}
Given a Weihrauch problem $P$, we abuse notation by using the formula $A\in P$ to mean that $A$ is a computable instance of $P$, that is, $A=P(x)$ for some computable $x\in{\rm dom}(P)$.
\end{Notation}

In the following, we will say that a tree $T$ is homogeneous if its set of paths is homogeneous. It corresponds to $[T]$ being some $\prod_{n\in\N}A_n$, that is $[T]$ is truly an instance of the axiom of choice. We see a homogeneous tree $T$ as a tree where the set $\{n\in\Nb:\sigma\fr n\in T\}$ does not depend on $\sigma\in T$.

Before going further, we mention that under Medvedev reducibility, AC and DC are always different, as there exists product of two homogeneous set that are never Medvedev equivalent to a homogeneous set.

\begin{Proposition}\label{prop:notEquiv}
  For every $\star\in\{\sf fin, cof, foc, aof\}$, there exists $A\in\dc{\star}$ such that there is no $B\in\ac{\star}$ with $A\equiv_M B$.
\end{Proposition}
\begin{proof}
  Simply take $A_0$ and $A_1$ in $\ac{fin}$ with are not Medvedev equivalent, and consider $C=0\fr A_0\cup1\fr A_1$, which is in $\dc{fin}$.
  Now, toward a contradiction, suppose also that there exists $H$ in $\ac{}$ (actually there is no need for $H$ to be $\Sigma^1_1$) such that $C\equiv_MH$.
  Let $\phi$ and $\psi$ be witness of this, i.e $\phi$ (resp. $\psi$) is total on $C$ (resp. $H$) and its image is included in $H$ (resp. $C$).
  
  Now, we describe a way for some $A_{i}$ to Medvedev compute $A_{1-i}$: Let $i\in 2$ and $\sigma$ be extensible in $H$ such that $\psi(\sigma;0)=1-i$. Given $x\in A_{i}$, apply $\phi$ on $i\fr x$ to obtain an element $y$ of $H$. Replace the beginning of $y$ by $\sigma$ and apply $\phi$: by homogeneity, $y$ with $\sigma$ as beginning is still in $H$, and the result has to be in $(1-i)\fr A_{1-i}$.

  For other values of $\star$, the proof is very similar.
\end{proof}

Note that the above proof used the fact that there always exists infinum in $\dc{\star}$ while this is not clear in $\ac{\star}$. However, using Weihrauch reducibility, dependent and independent choices are equivalent:
\begin{Theorem}
  $\ac{fin}\equiv_W\dc{fin}$
\end{Theorem}
\begin{proof}
  It is clear by Fact~\ref{KMP} that we have $\ac{fin}\leq_W\dc{fin}\leq_W\Sigma^1_1\mbox{-}{\sf WKL}\leq_W\widehat{\Sigma^1_1\mbox{-}{\sf C}_2}\leq_W\ac{fin}$.
\end{proof}

\subsection{The Medvedev lattices of $\ac{fin}$ and $\dc{fin}$}

In this section we examine the Medvedev degree structure of
$\Sigma^1_1$ choice for finite sets. We already have defined the
compact choice $\Sigma^1_1\mbox{-}{\sf KC}_{\N^\N}$ in
Section~\ref{subsection:compact-choice}, which is clearly the same problem as
$\dc{fin}$ up to the coding of the instance. In
Proposition~\ref{prop:compact-prouduct-of-two} we proved that for dependant
choice, the finite choice can always be weakened to independent
choice over 2 possibility, making $\ac{fin}=_{\sf W}\dc{fin}$.

In the following, we are interested in a finer analysis of $\ac{fin}$
and $\dc{fin}$ using Medvedev reducibility. In particular, we show
that upward density does not hold in both of these lattices: Indeed,
we show that there is a single nonempty compact homogeneous
$\Sigma^1_1$ set coding all information of nonempty compact
$\Sigma^1_1$ sets.  This can be viewed as an effective version of
Dellacherie's theorem (cf.\ Steel \cite{Ste80}) in descriptive set
theory.

\begin{Theorem}\label{th:fin-maximum}
  There exists a maximum in the Medvedev lattices of $\ac{fin}$ and in
  $\dc{fin}$. In other words, there exists $A\in\ac{fin}$ such that
  for every $B\in\dc{fin}$, $B\leq_M A$.
\end{Theorem}
\begin{proof}
  To construct a greatest element in $\dc{fin}$, we only need to
  enumerate all nonempty compact $\Sigma^1_1$ sets
  $S_e\subseteq\om^{\om}$.  Consider a $\Delta^1_1$ approximation
  $(S_{e,\alpha})_{\alpha<\omCK}$ of $S_e$.  Note that emptiness of
  $S_e$ is a $\Pi^1_1$-property, and therefore, if $S_e=\emptyset$,
  then it is witnessed at some stage $\alpha<\omCK$.  Let $\alpha$ be
  the least ordinal such that $S_{e,\alpha}$ is empty.  By compactness
  of $S_e$, such an $\alpha$ must be a successor ordinal.

  Now we construct a uniform sequence $(T_e)_{e\in\om}$ of nonempty
  $\Sigma^1_1$ sets such that if $S_e\not=\emptyset$ then $S_e=T_e$.
  Define $T_{e,0}=\N^\N$, and for any $\alpha>0$,
  $T_{e,\alpha}=S_{e,\alpha}$ if $S_{e,\alpha}\not=\emptyset$.  If
  $\alpha>0$ is the first stage such that $S_{e,\alpha}=\emptyset$,
  then $\alpha$ is a successor ordinal, say $\alpha=\beta+1$, and
  define $T_{e,\gamma}=T_{e,\beta}$ for any $\gamma\geq\alpha$, and
  ends the construction.  It is not hard to check that the sequence
  $(T_e)_{e\in\om}$ has the desired property.

  As a maximal element, it suffices to take the product of all
  $T_n$. Note that by the fact that $\ac{fin}\equiv_W\dc{fin}$, it
  also shows the maximality result for $\ac{fin}$.
\end{proof}

Even if lattices of dependent and independent choice share a common
maximum, they still have structural differences. The most evident one
is the existence of infinums: Given two $\Sigma^1_1$ trees $T_1$ and
$T_2$, it is easy to create a tree $T$ such that $[T]$ is the infinum
of $[T_0]$ and $[T_1]$, by considering for example
$0\fr T_0\cup1\fr T_1$, or $2\N\fr T_0\cup (2\N+1)\fr T_1$ depending
on the restriction on the dependent choice. However, this is not
possible when the trees are homogeneous as in the independent
choice. We now prove that $\ac{fin}$ has infinum for pairs in
$\ac{aof}$, by first showing that below any $\Sigma^1_1$ compact set,
there is a greatest homogeneous degree.

\begin{Theorem}\label{th:infinum}
  For every $A$ in $\dc{fin}$, there exists $X\leq_MA$ in $\ac{aof}$
  such that:
  \[\forall Y\in\ac{}\ [Y\leq_MA\implies Y\leq_M X].\]
\end{Theorem}
\begin{proof}
  We will define $X$ to be equal to $\prod_e\prod_nS^e_n$, with the
  following requirement: if $\phi_e$ is total on $A$, then
  $\prod_nS^e_n$ is included in the smallest homogeneous superset of
  $\Phi_e(A)$.  $S^e_n$ is defined by the following $\Sigma^1_1$ way:
  First wait to see that $\Phi_e$ is total on $A$. If it happens, and
  wait for $\Phi_e(A;n)$ to be finite, which has to happen by
  compactness of $A$. Then, remove everything but the values
  $\Phi_e(A;n)$.

  To conclude, if $Y\in\ac{}$ is such that $Y\leq_MX$ as witnessed by
  $\Phi_e$, then $Y\in\leq_MX$ as $\prod_nS^e_n\subseteq Y$.
\end{proof}
\begin{Corollary}
  For every $A,B\in\ac{fin}$, there exists $X\in\ac{aof}$ such that
  $X\leq_MA,B$ and
  \[\forall Y\in\ac{}\ [(Y\leq_MA\ \land\ Y\leq_M B)\implies Y\leq_M X].\]
\end{Corollary}
\begin{proof}
  Just apply Theorem~\ref{th:infinum} to $0\fr A\cup1\fr B$.
\end{proof}

As a special property of $\Sigma^1_1$ compact sets, we have the following analog
of the hyperimmune-free basis theorem.  For $p,q\in\N^\N$ we say that
$p$ is {\em higher Turing reducible to $q$} (written $p\leq_{hT}q$) if
there is a partial $\Pi^1_1$-continuous function
$\Phi\pcolon\N^\N\to\N^\N$ such that $\Phi(q)=p$ (see
Bienvenu-Greenberg-Monin \cite{BGM17} for more details).

\begin{Lemma}\label{lemma:higher-hyperimmunefree}
  For any $\Sigma^1_1$ compact set $K\subseteq \N^\N$ there is an
  element $p\in K$ such that every $f\leq_{hT}p$ is majorized by a
  $\Delta^1_1$ function.
\end{Lemma}

\begin{proof}
  Let $(\psi_e)$ be a list of higher Turing reductions.  Let $K_0=K$.
  For each $e$, let $Q_{e,n}=\{x\in \N^\N:\psi_e^x(n)\uparrow\}$.
  Then $Q_{e,n}$ is a $\Sigma^1_1$ closed set.  If $K_e\cap Q_{e,n}$
  is nonempty for some $n$, define $K_{e+1}=K_e\cap Q_{e,n}$ for such
  $n$; otherwise define $K_{e+1}=K_e$.  Note that if $K_e\cap Q_{e,n}$
  is nonempty for some $n$, then $\psi_e^x$ is undefined for any
  $x\in K_{e+1}$.  If $K_e \cap Q_{e,n}$ is empty for all $n$, then
  $\psi_e$ is total on the $\Sigma^1_1$ compact set $K_e$, one can
  find a $\Delta^1_1$ function majorizing $\psi_e^x$ for all
  $x\in K_e$ (cf.\ \cite{KMP}).  Define $K_\infty=\bigcap_nK_n$, which
  is nonempty.  Then, for any $p\in K_\infty$, every $f\leq_{hT}p$ is
  majorized by a $\Delta^1_1$ function.
\end{proof}

Note that continuity of higher Turing reduction is essential in the
above proof.  Indeed, one can show the following:

\begin{Proposition}
  There is a nonempty $\Sigma^1_1$ compact set $K\subseteq \N^\N$ such
  that for any $p\in K$, there is $f\leq_{T}p'$ dominates all
  $\Delta^1_1$ functions.
\end{Proposition}

\begin{proof}
  Let $(\varphi_e)$ be an effective enumeration of all partial
  $\Pi^1_1$ functions $\varphi_e\pcolon\om\to 2$.  As in the argument
  in Proposition \ref{prop:compact-choice-total-DNC} or Proposition
  \ref{prop:aouc-total-n}, one can see that the set $S_e$ of all
  two-valued totalizations of the partial $\Pi^1_1$ function
  $\varphi_e$ is nonempty and $\Sigma^1_1$.  Then the product
  $K=\prod_eS_e$ is also a nonempty $\Sigma^1_1$ subset of $2^\om$.
  It is clear that every $p\in K$ (non uniformly) computes any total
  $\Delta^1_1$ function on $\om$.  Let $BB$ be a total $p'$-computable
  function which dominates all $p$-computable functions.  In
  particular, $BB\leq_Tp'$ dominates all $\Delta^1_1$ functions.
\end{proof}

\subsection{The Medvedev lattices of $\ac{aof}$ and $\dc{aof}$}

We now discuss about choice, when the sets from which we choose can be
either everything, or finite. We will show that under the Weihrauch
scope, this principle is a robust one that is strictly above
$\dc{fin}$. It also share with the latter that dependent or
independent choice does not matter, and the existence of a maximal
element containing all the information, with very similar proof as for
$\dc{fin}$.

In Proposition~\ref{prop:aouc-total-n}, we showed that $\ac{aou}$ is robust. We
give two other evidences of this in the following theorems.
\begin{Theorem}
  For any $A\in\ac{aof}$, there exists $B\in\ac{aou}$ such that
  $A\leq_MB$.
\end{Theorem}
\begin{proof}
  Let $A=\prod_n A_n\in\ac{aof}$. We define
  $B=\prod_{\langle m,n\rangle} B^m_n$ such that $A\leq_MB$. We will
  ensure that there exists a single computable function $\Phi$ such
  that for any $m$ and $X\in\prod_nB^m_n$ we have $\Phi(X)\in A_m$.

  We first describe the co-enumeration of $B^m_n$. Let
  $(A_{m,\alpha})_{\alpha<\omCK}$ be an approximation of $A_m$. First,
  wait for the first stage where $A_m$ is finite. If it happens, wait
  for exactly $n$ additional elements to be removed from $A_m$. If
  this happens, remove from $B^m_n$ all elements but $c\in\N$, the
  integer coding for the finite set $A_m$ at this stage, say at stage
  $\alpha_n$.  More formally, let $D_e$ be the finite set coded by
  $e$, and set $B^m_n=\{c\}$ with $D_c=A_m[\alpha_n]$.

  Now, we describe the function $\Phi$. Given $X$, find the first $i$
  such that we do not have the following: $X(i+1)$ viewed as coding a
  nonempty finite set consists of elements from $X(i)$ with exactly
  one element removed.  Note that $X(0)$ codes a finite set, so the
  length of chains $D_{X(0)}\supsetneq D_{X(1)}\supsetneq\dots$ has to
  be finite.  Therefore, there exists such an $i$.  Then, output any
  element from $D_{X(i)}$.  Whenever we reach stage $\alpha_n$, we
  have $D_{X(n)}=A_m[\alpha_n]$, and thus $i\geq n$.  This implies
  that the chosen element $\Phi(X)$ is contained in $A_m$, as
  required.
\end{proof}

We have seen in Proposition \ref{prop:compact-choice-total-DNC} that
$\ac{fin}$ is Weihrauch equivalent to $\Pi^1_1\mbox{-}{\sf Tot}_2$ and
$\Pi^1_1\mbox{-}{\sf DNC}_2$.  Moreover, we have also shown in
Proposition \ref{prop:aouc-total-n} that $\ac{aou}$ is Weihrauch
equivalent to $\Pi^1_1\mbox{-}{\sf Tot}_\N$.  Recall from Section
\ref{sec:restricted-choice-principles} we have introduced the
$\Pi^1_1$-diagonalization principle $\Pi^1_1\mbox{-}{\sf DNC}_\N$,
which is a special case of the cofinite (indeed, co-singleton)
$\Sigma^1_1$-choice principle.  In particular, at first we know a
bound of the number of elements removed from a cofinite set.  We now
consider the following principle for a bound $\ell\in\N$:
\[
\Sigma^1_1\mbox{-}{\sf C}^{{\sf cof}\upto \ell}_X=\Sigma^1_1\mbox{-}{\sf C}_{X}\rest{\{A\subseteq X:|X\setminus A|\leq \ell\}}.
\]

We call the coproduct of $(\Sigma^1_1\mbox{-}{\sf C}^{{\sf cof}\upto \ell}_{X})_{\ell\in\N}$ the {\em strongly-cofinite choice} on $\N$, and write $\Sigma^1_1\mbox{-}{\sf C}^{{\sf cof}\upto \ast}_{X}$.
Later we will show that the cofinite choice $\ac{cof}$ is not Medvedev or Weihrauch reducible to the all-or-finite choice $\ac{aof}$; however we will see that the strong cofinite choice is Medvedev/Weihrauch reducible to $\ac{aof}$.

Even more generally, we consider the {\em finite-or-strongly-cofinite choice}, denoted $\ac{fosc}$, which accepts an input of the form $(p,\psi)$, where for any $n\in\N$, $p(n)$ is a code of a $\Sigma^1_1$ subset $S_{p(n)}$ of $\N$ such that either $S_{p(n)}$ is nonempty and finite, or $|\N\setminus S_{p(n)}|\leq \psi(n)$.
If $(p,\psi)$ is an acceptable input, then $\ac{fosc}$ chooses one element from $\prod_nS_{p(n)}$.

We show that the all-or-unique choice is already strong enough to compute the finite-or-strongly-cofinite choice:

\begin{Proposition}
 For any $A\in\ac{fosc}$, there exists $B\in\ac{aof}$ such that $A\leq_MB$.
\end{Proposition}

\begin{proof}
Let $A=\prod_nA_n$ with a bound $\psi$ is given.
We will construct a uniformly $\Sigma^1_1$ sequence $(B^n_m)_{m\leq\psi(n)}$ of subsets of $\N$.
We use $B^{n}_{0},B^{n}_{1},\dots,B^{n}_{\psi(n)-1}$ to code information which element is removed from $A_n$ whenever $A_n$ is cofinite, and use $B^{n}_{\psi(n)}$ to code full information of $A_n$ whenever $A_n$ is finite.
If $a_0$ is the first element removed from $A_n$, then put $B^n_0=\{a_0\}$, and if $a_1$ is the second element removed from $A_n$, then put $B^n_1=\{a_1\}$, and so on.
If $A_n$ becomes a finite set, then $B_{\psi(n)}$ just copies $A_n$.
One can easily ensure that for any $n\in\N$ and $m<\psi(n)$, if $A_n$ is finite, then $B^n_m$ is a singleton, which is not contained in $A_n$; otherwise $B^n_m=\N$.
Moreover, we can also see that either $B^n_{\psi(n)}$ is nonempty and finite or $B^n_{\psi(n)}=\N$.

Now, assume that $X\in\prod_{n,m}B^n_m$ is given.
If $X(n,\psi(n))\not\in\{X(n,i):i<\psi(n)\}$, then put $Y(n)=X(n,\psi(n))$.
Otherwise, choose $Y(n)\not\in\{X(n,i):i<\psi(n)\}$.
Clearly, the construction of $Y$ from $X$ is uniformly computable.

If $A_n$ becomes a finite set, the first case happens, and $Y(n)=X(n,\psi(n))\in B^n_{\psi(n)}=A_n$.
If $A_n$ remains cofinite, it is easy to see that $\N\setminus A_n\subseteq\{X(n,i):i<\psi(n)\}$, and therefore $Y(n)\in A_n$.
Consequently, $Y\in A$.
\end{proof}

\begin{Corollary}
  $\ac{aou}\equiv_{\sf W}\ac{aof}\equiv_{\sf W}\ac{fosc}$.
\end{Corollary}

In the following we will only consider all-or-finite choice, by convenience. We now prove that dependent choice does not add any power, and the existence of a maximal instance that already code all the other instances, with very similar proofs as in the $\dc{fin}$ case.

\begin{Theorem}\label{th:equ-ac-dc-aof}
For every $A\in\dc{aof}$ there exists $B\in\ac{aof}$ such that $A\leq_MB$.
\end{Theorem}
\begin{proof}
  The argument is similar as for the finite case (Fact \ref{KMP} or
  Theorem \ref{th:fin-maximum}). If $T$ is a $\Sigma^1_1$ tree, define
  $T_\sigma$ by the following $\Sigma^1_1$ procedure: First, wait for
  $\{n:\sigma\fr n\in T\}$ to be finite but nonempty. If this happens,
  at every stage define $T_\sigma$ to be $\{n:\sigma\fr n\in T\}$
  except if this one becomes empty.  Note that if
  $\{n:\sigma\fr n\in T\}$ becomes a finite set at some stage
  $\alpha_0$, but an empty set at a later stage $\alpha_1$, then the
  least such stage $\alpha_1$ must be a successor ordinal, and therefore
  we can keep $T_\sigma$ being nonempty (see also the proof of Theorem
  \ref{th:fin-maximum}).  Clearly, $T_\sigma$ is either finite or $\N$
  and $\prod_{\sigma\in\om^{<\om}} [T_\sigma]\geq_M[T]$.
\end{proof}
\begin{Corollary}
  $\ac{aof}\equiv_W\dc{aof}$.
\end{Corollary}
\begin{proof}
  By uniformity of the precedent proof.
\end{proof}

The upward density of the axiom of choice on ``all-or-finite'' sets would allow us to Weihrauch separate it from its ``finite'' version. However, $\ac{aof}$ does also have a maximum element.

\begin{Theorem}\label{th:max-aof}
  There exists a single maximum Medvedev degree in $\ac{aof}$ and $\dc{aof}$.
\end{Theorem}
\begin{proof}
  The argument is similar as Theorem \ref{th:fin-maximum}, even though
  we have no compactness assumption.  By the fact that
  $\ac{aof}\equiv_W\dc{aof}$, it suffices to prove the result for one,
  let's say $\ac{aof}$. Let $A_e=\prod_n S^e_n$ be the $e$-th
  $\Sigma^1_1$ homogeneous set.  We set
  $\widehat {A_e}=\prod_n\widehat {S^e_n}$ to be defined by the
  following $\Sigma^1_1$ procedure: First, wait for some $S^e_n$ to
  become finite and nonempty. If this happens, define
  $\widehat{S^e_n}=S^e_n$ until it removes its last element. At this
  point, leaves $\widehat {S^e_n}$ nonempty, which is possible since
  it can happen only at a successor stage (see also the proof of
  Theorem \ref{th:equ-ac-dc-aof}).

  Then $(\widehat {A_e})_{e\in\N}$ is an enumeration of all nonempty
  elements of $\ac{aof}$. Define the maximum to simply be
  $\prod_e\prod_n\widehat{S^e_n}$.
\end{proof}

We now prove that the relaxed constraint on the sets that allows them
to be full does increase the power of the choice principle, making
$\ac{aof}$ strictly above $\ac{fin}$. We use the fact that the lattice
of $\ac{fin}$ has a maximal element, and we show that it must be
strictly below some instance of $\ac{aof}$.

\begin{Theorem}\label{th:aof-sup-fin}
  For every $A\in\ac{fin}$, there exists $B\in\ac{aof}$ such that $A<_MB$.
\end{Theorem}
\begin{proof}
  We will find $C=\prod_nC_n\in\ac{aof}$ such that
  $C\not\leq_MA$. Then, $A\times C$ will witness the theorem.

  Now, let us describe the co-enumeration of $C_n$. First, wait for
  $\Phi_n(\cdot;n)$ to be total on $A$, where $\Phi_n$ is the $n$-th
  partial computable function.
  Then, wait for it to take only finitely many values, which will
  happen by compactness. At this point, remove everything from $C_n$
  except $\max\Phi_n(A;n)+1$.

  We have that $C_n$ is either $\N$ if the co-enumeration is stuck
  waiting for $\Phi_n(\cdot;n)$ to be total, or a singleton
  otherwise. Also, it is clear that for any $n$, $\Phi_n$ cannot be a
  witness that $C\leq_MA$, so $C\not\leq_MA$.
\end{proof}

\begin{Corollary}
  We have $\ac{fin}<_W\ac{aof}$.
\end{Corollary}
\begin{proof}
  By Theorem~\ref{th:aof-sup-fin} and Theorem~\ref{th:fin-maximum}.
\end{proof}

One can also use the domination property to separate the all-or-finite choice principle and the ($\sigma$-)compact principle.

\begin{Proposition}\label{prop:aof-domination}
There exists $A\in\ac{aof}$ such that every element $p\in A$ computes a function which dominates all $\Delta^1_1$ functions.
\end{Proposition}

\begin{proof}
Let $(\varphi_e)_{e\in\N}$ be an effective enumeration of all partial $\Pi^1_1$ functions on $\om$.
Put $s(e)=\sum_{n\leq e}n$.
Define $A_{s(e)+k}\subseteq\N$ for $k\leq e$ as follows.
Begin with $A_{s(e)+k}=\N$.
Wait until we see $\varphi_e(k)\downarrow$.
If it happens, set $A_{s(e)+k}=\{\varphi_e(k)\}$.
Define $A=\prod_nA_n$.
Then define $\Psi(p;n)=\sum_{k\leq e}p(k)$, which is clearly computable in $p$.
It is easy to see that $\Psi(p)$ dominates all $\Delta^1_1$ function whenever $p\in A$.
Indeed, since $\Psi$ is total, every $p\in A$ $tt$-computes a function which dominates all $\Delta^1_1$ functions.
\end{proof}

This shows that the all-or-finite $\Sigma^1_1$-choice is not
Weihrauch-reducible to the $\sigma$-compact $\Sigma^1_1$-choice.

\begin{Corollary}\label{cor:aof-versus-sigma-compact}
$\ac{aof}\not\leq_{\sf W}\Sigma^1_1\mbox{-}{\sf K}_\sigma{\sf C}_{\N^\N}$.
\end{Corollary}

\begin{proof}
Recall that a computable instance of $\Sigma^1_1\mbox{-}{\sf K}_\sigma{\sf C}_{\N^\N}$ is a countable union of compact $\Sigma^1_1$ sets.
Thus, by Lemma \ref{lemma:higher-hyperimmunefree}, there is a solution $p$ to a given computable instance of $\Sigma^1_1\mbox{-}{\sf K}_\sigma{\sf C}_{\N^\N}$ such that any function which is higher Turing reducible to $p$ is majorized by a $\Delta^1_1$ function.
However, by Proposition \ref{prop:aof-domination}, there is a computable instance of $\ac{aou}$ whose solution consists of $\Delta^1_1$ dominants. 
\end{proof}

\begin{Corollary}
The $\sigma$-compact choice $\Sigma^1_1\mbox{-}{\sf K}_\sigma{\sf C}_{\N^\N}$ is not parallelizable, and $\Sigma^1_1\mbox{-}{\sf K}{\sf C}_{\N^\N}<_{\sf W}\Sigma^1_1\mbox{-}{\sf K}_\sigma{\sf C}_{\N^\N}$.
\end{Corollary}

\begin{proof}
  Clearly, $\ac{}$ (and therefore $\ac{aof}$) is Weihrauch reducible
  to the parallelization of the $\sigma$-compact $\Sigma^1_1$-choice
  $\Sigma^1_1\mbox{-}{\sf K}_\sigma{\sf C}_{\N^\N}$.  Therefore, by
  Corollary \ref{cor:aof-versus-sigma-compact}, the $\sigma$-compact
  $\Sigma^1_1$-choice is not parallelizable.  By definition, any
  $\ac{\star}$ is parallelizable, and so is the compact
  $\Sigma^1_1$-choice $\Sigma^1_1\mbox{-}{\sf K}{\sf C}_{\N^\N}$ by
  Proposition \ref{prop:compact-prouduct-of-two}.
\end{proof}

\subsection{The Medvedev lattices of $\ac{cof}$ and $\dc{cof}$}

The choice problem when all sets are cofinite is quite different from
the other restricted choices we study. It is the only one that does
not contains $\ac{fin}$.

Let us fix an instance $A=\prod_n A_n$ of $\ac{cof}$. For every $n$,
$A_n$ is cofinite, so there exists $a_n$ such that for any
$i\geq a_n$, we have $i\in A_n$. Now, call $f:n\mapsto a_n$. We have
that $f\in A$, and for every $g$ pointwise above $f$, we must have
$g\in A$. So we clearly have
$A\leq_{\sf W}\{g\in\om^\om:\forall i, f(i)\leq g(i)\}=A_f$. This
essential property of $\ac{cof}$ prevents an instance to have more
computational power than an $A_f$ for some $f\in\om^\om$.

The cofiniteness still allows some more power, as we will prove in
this section that $\ac{cof}$ is Weihrauch incomparable with both
$\ac{fin}$ and $\ac{aof}$.

\begin{Theorem}\label{th:cof-not-<-aof}
  There exists an $A\in\ac{cof}$ such that for any $B\in\ac{aof}$
  $A\not\leq_M B$.
\end{Theorem}
\begin{proof}
  We use the existence of a maximal all-or-finite degree of
  Theorem~\ref{th:max-aof} to actually only prove
  \[\forall B\in\ac{aof},\exists A\in\ac{aoc}: B\not\leq_M A.\]
  
  Fix a $B=\prod_{n\in\N}B_n$, with $B_n\subseteq\N$ being either $\N$
  or finite. We will construct $A=\prod_{e\in\Nb}S_e$, and use $S_e$
  to diagonalize against $\Phi_e$ being a witness for the reduction,
  by ensuring that either $\Phi_e$ is not total on $B$, or
  $\exists k\in\N, \sigma\in\prod_{n<k}B_n$ with
  $\Phi_e(\sigma;e)\downarrow\not\in S_e$. Here is a description of
  the construction of $S_e$, along with sequences of string
  $(\sigma_n)$ and $(\tau_n)$:
  \begin{enumerate}
  \item First of all, wait for a stage where $B\subseteq\dom(\Phi_e)$,
    that is $\Phi_e$ is total on the the current approximation of
    $B$. Define $\sigma_0=\epsilon=\tau_0$.
  \item Let $n$ be the maximum such that $\tau_n$ is defined. Find
    $\sigma_{n+1}\succ\tau_n$ such that
    $\Phi_e(\sigma_{n+1};e)\downarrow\in S_e$. Take $\sigma_{n+1}$ to
    be the least such, and remove $\Phi_e(\sigma_{n+1};e)$ from $S_e$.
  \item Wait for some stage where $\Phi_e(B;e)\subseteq S_e$. If it
    happens, wait again for the current approximation of $B$ to be
    ``all or finite'', which will happen. Take $\tau_{n+1}$ to be the
    greatest prefix of $\sigma_{n+1}$ still in $B$, and return to step
    (2).
  \end{enumerate}
  Let us prove that $S_e$ is cofinite. If the co-enumeration of $S_e$
  stays at step (1), then $S_e=\N$ is cofinite. Otherwise, let us
  prove that there can only be finitely many $\tau_n$ defined.

  Suppose infinitely many $(\tau_n)$ are defined. Then, this must have
  a limit: Let $l$ be a level such that $(\tau_n(l'))_n$ stabilizes
  for all $l'<l$. Start from a stage where they have stabilized. From
  this stage, if $\tau_n(l)$ change, it must have been removed from
  $B_l$. But then, $B_l$ will become finite before the co-enumeration
  continue, and $(\tau_n(l))$ can only take value from $B_l$ and never
  twice the same. Therefore, $(\tau_n(l))$ becomes constant at some
  point.

  If there are only finitely many $\tau_n$, then only finitely many
  things are removed from $S_e$ which is cofinite. It remains to prove
  that $B\not\leq_M A$. Suppose $\Phi_e$ is a potential witness for
  the inequality. Either $\Phi_e$ is not total on $B$, or we get stuck
  at some step in the co-enumeration of $S_e$, waiting for
  $\Phi_e(A;e)\in S_e$ to never happen, leaving us with
  $\Phi_e(A;e)\not\subseteq S_e$.
\end{proof}
\begin{Theorem}\label{th:ABDelta11path}
  For any $A\in\ac{aof}$ and $B\in\ac{cof}$, if $A\leq_MB$, then $A$ contains a $\Delta^1_1$ path.
\end{Theorem}
\begin{proof}
  Assume that $A\leq_MB$ via some functional $\Phi$, and $A$ and $B$
  are of the forms $\prod_nA_n$ and $\prod_nB_n$, respectively.  We
  describe the $\Delta^1_1$ procedure to define $C$:

  Given $n$, in parallel, wait for $n$ to be enumerated in one of
  those two $\Pi^1_1$ sets:
  \begin{enumerate}
  \item If $n$ is enumerated in
    $\{n:\exists k\in\N\forall f\in\Baire, \exists\sigma\geq f,\
    \Phi(\sigma;n)=k\}$, define $C(n)$ to be one of these $k$.
  \item If $n$ is enumerated in
    $\{n:\forall f\in\Baire, \forall k, \exists k'>k,
    \exists\sigma\geq f$ such that $\Phi(\sigma;n)=k'\}$ then define
    $C(n)=0$.
  \end{enumerate}
  Here, $\sigma\geq f$ denotes the pointwise domination order, that
  is, $\sigma(n)\geq f(n)$ for all $n<|\sigma|$.  It is clear that one
  of the two options will happen. Let $f\in\Baire$ be such that
  $\forall k\geq f(n), k\in B_n$. In case (1), it is clear that
  $C(n)\in A_n$. In case (2), it is clear that $A_n$ is infinite,
  therefore it is equal to $\N$ and $C(n)\in A_n$. So $C\in A$.
\end{proof}
\begin{Corollary}
  We have both $\ac{cof}\not\leq_W\ac{aof}$ and $\ac{fin}\not\leq_W\ac{cof}$.
\end{Corollary}
\begin{proof}
  The first part is implied by Theorem~\ref{th:cof-not-<-aof}. The second
  part is implied by Theorem~\ref{th:ABDelta11path} and the fact that there
  exists $\Sigma^1_1$ finitely branching homogeneous trees with no
  $\Delta^1_1$ member.
\end{proof}

We now show upper density of $\ac{cof}$, using a similar proof from Theorem~\ref{th:cof-not-<-aof}.

\begin{Theorem}\label{th:cof-upward-dense}
  The Medvedev degrees of $\ac{cof}$ are upward dense.
\end{Theorem}
\begin{proof}
  Fix a $B=\prod_{n\in\N}B_n$, with $B_n\subseteq\N$ being
  cofinite. We will construct $A=\prod_{e\in\Nb}S_e$, and use $S_e$ to
  diagonalize against $\Phi_e$ being a witness for the reduction, by
  ensuring that either $\Phi_e$ is not total on $B$, or
  $\exists k\in\N, \sigma\in\prod_{n<k}B_n$ with
  $\Phi_e(\sigma;e)\downarrow\not\in S_e$. Here is a description of
  the construction of $S_e$, along with sequences of string
  $(\sigma_n)$ and $(\tau_n)$:
  \begin{enumerate}
  \item First of all, wait for a stage where $B\subseteq\dom(\Phi_e)$,
    that is $\Phi_e$ is total on the the current approximation of
    $B$. Define $\sigma_0=\epsilon=\tau_0$.
  \item Let $n$ be the maximum such that $\tau_n$ is defined. Find
    $\sigma_{n+1}\succ\tau_n$ such that
    $\Phi_e(\sigma_{n+1};e)\downarrow\in S_e$. Take $\sigma_{n+1}$ to
    be the least such, and remove $\Phi_e(\sigma_{n+1};e)$ from $S_e$.
  \item Wait for some stage where $\Phi_e(B;e)\subseteq S_e$. Take
    $\tau_{n+1}$ to be the greatest prefix of $\sigma_{n+1}$ still in
    $B$, and return to step (2).
  \end{enumerate}
  Let us prove that $S_e$ is cofinite. If the co-enumeration of $S_e$
  stays at step (1), then $S_e=\N$ is cofinite. Otherwise, let us
  prove that there can only be finitely many $\tau_n$ defined.

  Suppose infinitely many $(\tau_n)$ are defined. Then, this must have
  a limit: Let $l$ be a level such that $(\tau_n(l'))_n$ stabilizes
  for all $l'<l$. Start from a stage where they have stabilized. From
  this stage, if $\tau_n(l)$ change, it must have been removed from
  $B_l$. But that can happen only finitely many times, as $B_l$ id
  cofinite. Therefore, $(\tau_n(l))$ becomes constant at some point.

  If there are only finitely many $\tau_n$, then only finitely many
  things are removed from $S_e$ which is cofinite. It remains to prove
  that $B\not\leq_M A$. Suppose $\Phi_e$ is a potential witness for
  the inequality. Either $\Phi_e$ is not total on $B$, or we get stuck
  at some step in the co-enumeration of $S_e$, waiting for
  $\Phi_e(A;e)\in S_e$ to never happen, leaving us with
  $\Phi_e(A;e)\not\subseteq S_e$.
\end{proof}

We here also note some domination property of the cofinite choice.
The following fact is implicitly proved by Kihara-Marcone-Pauly \cite{KMP} to separate $\Sigma^1_1\mbox{-}{\sf WKL}$ and $\widehat{\Sigma^1_1\mbox{-}{\sf C}_\N}$.

\begin{Fact}[\cite{KMP}]
There exists $A\in\ac{cof}$ such that every element $p\in A$ computes a function which dominates all $\Delta^1_1$ functions.
\end{Fact}

Therefore, as in the proof of Corollary \ref{cor:aof-versus-sigma-compact}, we can observe the following.

\begin{Corollary}
$\Sigma^1_1\mbox{-}{\sf AC}^{\sf cof}_{\N^\N}\not\leq_{\sf W}\Sigma^1_1\mbox{-}{\sf K}_\sigma{\sf C}_{\N^\N}$.
\end{Corollary}

\subsection{The Medvedev lattices of $\ac{foc}$, $\dc{foc}$, $\ac{}$, $\dc{}$}

In this part, we study the weakened restriction to sets that are
either finite, or cofinite. This restriction allows any instance from
the stronger restrictions, thus $\ac{aof}$, $\ac{fin}$, and $\ac{cof}$
are Weihrauch reducible to $\ac{foc}$ (and similarly for dependent
choice). It is the weakest form of restriction other than ``no
restriction at all'' that we will consider. However, we don't know if
this restriction does remove some power and is strictly below $\ac{}$
or not, as asked in Question~\ref{qu:sep-foc-hom}.

In the following, we will show upper density for both $\ac{foc}$,
$\dc{foc}$ and $\ac{}$, $\dc{}$. We will give several different proofs
of this result. Theorem~\ref{th:attempt} has a weaker conclusion, but is an
attempt to answer Question~\ref{qu:sep-foc-hom}. This attempt fails, by being
not effective enough to make a diagonalization out of it.

\begin{Theorem}\label{th:attempt}
  For every $A\in\ac{foc}$, there exists $B\in\ac{}$ such that $B\not\leq_MA$.
\end{Theorem}
\begin{proof}
  We will build $B=\prod_e B_e\in\ac{}$ by defining $B_e$ in a uniform
  $\Sigma^1_1$ way, such that if $\Phi_e$ is total on $A$, then
  $\Phi_e(A;e)\not\in B_e$.

  Fix $e\in\N$, and $A=\prod_n A_n\in\ac{foc}$. In our definition of
  the co-enumeration of $B_e$ along the ordinals, there will be two
  main steps in the co-enumeration: The first one forces that if
  $\Phi_e(A;e)\subseteq B_e$, then for every $l$,
  $|\Phi_e(A_{\upharpoonright\leq l})|<\omega$ where
  $A_{\upharpoonright\leq l}=\{\sigma\in\om^{\leq l}:[\sigma]\cap
  A\neq\emptyset\}$. The second step will force that if
  $\Phi_e(A;e)\subseteq B_e$, then $A$ is empty or $\Phi_e$ is not
  total on $A$.

  In order to conduct all these steps, we will need to remove several
  times an element of $B_e$, but we do not want it to become
  empty. This is why in parallel of removing elements from $B_e$, we
  also mark some as ``saved for later'', so we know that even after
  infinitely many removal, $B_e$ is still
  infinite.

  We now describe the first part of the co-enumeration. For clarity,
  we use the formalism of an infinite time algorithm, that could
  easily be translated into a $\Sigma^1_1$ formula.

\begin{algorithm}[H]
  \For{$l\in\omega$}{
    Mark a new element of $B_e$ as saved\;
    \While{$\Phi_e(A_{\upharpoonright\leq l};e)$ is infinite} {
      \For{$i\in\omega$} {
        Mark a new element of $B_e$ as saved\;
        Remove from $B_e$ the first element of $\Phi_e(A_{\upharpoonright\leq l};e)$ that is not saved, if it exists. Otherwise, exit the loop\;
        Wait for $\Phi_e(A;e)\subseteq B_e$\;
      }
      Wait for every $A_n$ to be finite or cofinite\;
      Unmark the elements marked as saved by the ``for $i\in\omega$'' loop\;
    }
  }
\end{algorithm}

Let us first argue that for a fixed $l$, the ``while'' part can only
be executed a finite number of times. At every execution of the ``for
$i\in\omega$'' loop, either one element of $A_{\upharpoonright\leq l}$
is removed, or $\Phi_e(A_{\upharpoonright\leq l};e)$ is finite and we
exit the while loop (this is because at every step, only finitely many
elements are marked as saved). But this means that if a ``for'' loop
loops infinitely many times, by the pigeon hole principle there must
exists a specific level $l_0\leq l$ such that $A_{l_0}$ went from
cofinite to finite. But this can happen only $l+1$ times, and the
``while'' loop can only run $l+1$ many times.

Let us now argue that at every stage of the co-enumeration, including
its end, $B_e$ is infinite. Fix a level $l$, and suppose that at the
beginning of a ``while'' loop, $B_e$ is infinite. As after every loop
of the ``for $i\in\om$'' loop one element is saved, it means that at
after all these infinitely many loop, $B_e$ contains infinitely many
elements. This will happen during only finitely many loops of the
``while'' loop, so at the beginning of level $l+1$, $B_e$ is
infinite. A similar argument with the elements saved by the first
``for $l\in\om$'' loop shows that if the first part of the
co-enumeration ends, $B_e$ is still infinite.

Now we split into two cases. If the first part of the co-enumeration
never stops, as the ``while'' loop is in fact bounded, it means that
the co-enumeration is forever stuck waiting for
$\Phi_e(A;e)\subseteq B_e$. But as this never happens, $B_e$ has the
required property. Otherwise, the first part of the co-enumeration
ends, and we are at a stage where for every $l$,
$\Phi_e(A_{\upharpoonright\leq l};e)$ is finite, but $B_e$ is
infinite. We now continue to the second part of the co-enumeration of
$B_e$:

\begin{algorithm}[H]
  \For{$l\in\omega$}{
    Remove from $B_e$ all the elements of $\Phi_e(A_{\upharpoonright\leq l};e)$\;
    Wait for $\Phi_e(A;e)\subseteq B_e$\;
  }
\end{algorithm}

We argue that this co-enumeration never finish. Let $x\in A$, and
$\sigma\prec x$ such that $\Phi_e(\sigma;e)\downarrow=k$. The
co-enumeration will never reach the stage where $l=|\sigma+1|$, as it
cannot go through $l=|\sigma|$: If it reaches such stage, it will
remove $k$ from $B_e$ and never have $\Phi_e(A;e)\subseteq B_e$. So,
the co-enumeration has to stop at some step of the ``for'' loop,
waiting for $\Phi_e(A;e)\subseteq B_e$ never happening. As $B_e$ is
infinite, it has the required property.
\end{proof}

In order to Weihrauch-separate $\ac{foc}$ from the unrestricted
$\ac{}$, one would need a stronger result with a single $B\in\ac{}$
not Medvedev reducible to any $A\in\ac{foc}$. We could try to apply
the same argument to define
$\prod_{\langle n,e\rangle}B_{\langle n,e\rangle}$, this time
diagonalizing against an enumeration $(S^e)_{e\in\N}$ of
$S^e=\prod_n S^e_n\in\ac{}$. If $S^e$ is not in $\ac{foc}$, the
co-enumeration will be stuck somewhere in the co-enumeration of some
level, with no harm to the global diagonalization.

However, if some particular $S^e$ is empty, we could end up with some
$B_{\langle n,e\rangle}=\emptyset$, making $B$ empty. Indeed, suppose
we reach the second part of the co-enumeration. Then, the malicious
$S^e$ can make sure that every step of the second loop are achieved,
by removing from $S^e$ all strings $\sigma$ such that
$\Phi_e(\sigma;e)\downarrow\not\in B_{\langle n,e\rangle}$, at every
stage of the co-enumeration. As a result, both $S^e$ and
$B_{\langle n,e\rangle}$ will become empty.

\begin{Question}\label{qu:sep-foc-hom}
  Do we have $\ac{foc}<_{\sf W}\ac{}$?
\end{Question}
We now give a stronger result with a much simpler, but not effective, proof. As a corollary, we will obtain the upper density of $\ac{}$ and $\dc{}$.

\begin{Theorem}\label{th:upper-density-strong}
  For every $A\in\dc{}$, there exists $B\in\ac{}$ such that $B\not\leq_MA$.
\end{Theorem}
\begin{proof}
  We first claim that there is no enumeration of all nonempty elements
  of $\ac{}$. More than that, we will prove that there is no
  $\prod_{n,e\in\N}S^e_n\in\ac{}$ uniformly $\Sigma^1_1$ such that for
  every $B=\prod_n B_n\in\ac{}$, there exists an $e$ such that
  $\prod_nS^e_n\subseteq B$.  Let $(S^e_n)_{n,e\in\N}$ be any
  uniformly $\Sigma^1_1$ enumeration. We construct $(B_e)_{e\in\N}$, a
  witness that this enumeration is not a counter-example to our
  claim. We define $B_e$ by stage: At stage $\alpha$, $B_e$ is equal
  to the interval $]\min(S^e_e);\infty[$, where $\min(S^e_e)$ is
  computed up to stage $\alpha$. This defines a $\Sigma^1_1$ set. We
  have $\prod_n B_n\not\supseteq\prod_nS^e_n$ for every $e\in\N$ and
  the claim is proven.

  Now, suppose that there exists $A\in\dc{}$ such that for every
  $B\in\ac{}$, we have $B\leq_MA$. Let us define $S^e_n$ by
  \[m\in S^e_n\Leftrightarrow\exists X\in
    A:\Phi_e(X;n)\downarrow=m\text{ or $\Phi_e$ is not total on
      $A$}.\]
  Given any $B\in\ac{}$, as $B\leq_M A$, fix a witness
  $\Phi_e$. We have $\Phi_e(A)\subseteq B$, and as $B$ is homogeneous
  we also have $\prod_nS^e_n\subseteq B$. Then, $(S^e_n)_{e,n\in\N}$
  would be a contradiction to our first claim.
\end{proof}

\begin{Corollary}
  We have upward density for $\ac{}$ and $\dc{}$.
\end{Corollary}

There is another non-effective proof showing upward density for $\dc{}$ (but not for $\ac{}$).
Indeed, remarkably, the result shows that there is no greatest nonempty $\Sigma^1_1$ closed set even with respect to hyperarithmetical Muchnik degrees.
We say that $A\subseteq\om^\om$ is {\em hyperarithmetically Muchnik reducible to $B\subseteq\om^\om$} (written $A\leq^{\sf HYP}_wB$) if for any $x\in A$ there is $y\in B$ such that $y\leq_hx$, that is, $y$ is hyperarithmetically reducible to $x$.

\begin{Fact}[Gregoriades {\cite[Theorem 3.13]{Gre}}]\label{fact:gregoriades}
If $P$ is a $\Delta^1_1$ closed set with no $\Delta^1_1$ element, then there exists a clopen set $C$ such that $P\cap C\not=\emptyset$ and $P<^{\sf HYP}_wP\cap C$.
\end{Fact}

Note that any $P$ satisfying the conclusion of the above fact cannot be homogeneous since if $P$ is homogeneous, $C$ is clopen, and $P\cap C$ is nonempty, then we always have $P\cap C\equiv_MP$.
So, Fact \ref{fact:gregoriades} does not imply Theorem \ref{th:upper-density-strong}.

\begin{Corollary}
For any nonempty $\Sigma^1_1$ set $A\subseteq\om^\om$, there is a nonempty $\Pi^0_1$ set $B\subseteq\om^\om$ such that $A<_w^{\sf HYP}B$.
\end{Corollary}

\begin{proof}
For any nonempty $\Sigma^1_1$ set $A$, it is easy to see that there is a nonempty $\Pi^0_1$ set $A^\ast$ such that $A\leq_MA^\ast$.
If $A^\ast$ has a $\Delta^1_1$ element, then the assertion is clear.
If $A^\ast$ has no $\Delta^1_1$ element, by Fact \ref{fact:gregoriades}, there is clopen $C$ such that $A\leq_MA^\ast<_w^{\sf HYP}A^\ast\cap C$.
\end{proof}

In \cite{densityPi01}, Cenzer and Hinman showed that the lattice of $\Pi^0_1$ classes in Cantor space is dense. Here we already showed upward density, we now prove downward density:

\begin{Theorem}\label{thm:lower-density}
  $\dc{}$ is downward dense. In other words, for every $A\in\dc{}$ with no computable member, there exists $B>_M\Baire$ in $\dc{}$ such that \[\Baire<_MA\cup B<_MA.\]
\end{Theorem}

\begin{proof}
We first reduce the problem to finding a non-computable hyperarithmetical real $X$ such that $A$ contains no $X$-computable point.
Indeed, for any computable ordinal $\alpha$, by assuming that $A$ has no $\emptyset^{(\alpha)}$-computable point, we construct a hyperarithmetical real $X\not\leq_T\emptyset^{(\alpha)}$ such that $A$ contains no $X^{(\alpha)}$-computable element.
If such an $X$ exists, then we have $\om^\om<^\alpha_MA\cup\{X\}<^\alpha_MA$, where $\leq_<M^\alpha$ indicates the Medvedev reducibility with the $\alpha$-th Turing jump.

   It suffices to show that $\Phi_e(X\oplus\emptyset^{(\alpha)})\not\in A$ for any $e$, and $\emptyset^{(\alpha)}<_TX\oplus\emptyset^{(\alpha)}\equiv_TX^{(\alpha)}$.
   The latter condition is ensured by letting $X$ be $\alpha$-generic.
    To describe a strategy for ensuring the first condition, fix a pruned $\Sigma^1_1$ tree $T_A$ such that $[T_A]=A$.
    Let $\Phi_e^\alpha$ be the $\emptyset^{(\alpha)}$-computable function mapping $Z$ to $\Phi_e(Z\oplus\emptyset^{(\alpha)})$.
    There are two ways for $\Phi_e^\alpha$ to not be a witness that $A$ has no $X$-computable element: either $\Phi_e^\alpha(\sigma)\not\in T_A$ for some $\sigma\prec X$, or $X\not\in\dom(\Phi_e^\alpha)$. 
    Let us argue that we have the following:
    For any $e\in\N$ and $\sigma\in\om^{<\om}$ there exists a finite string $\tau$ extending $\sigma$ such that
    \begin{equation}\label{disj}
\text{either $\Phi_e^\alpha(\tau)\not\in T_A$ or $[\tau]\cap\mathrm{dom}(\Phi_e^\alpha)=\emptyset$}
    \end{equation}
    Indeed, if it were not the case for some $e\in\omega$, we would have a string $\sigma$ such that for every $\tau$, $\Phi_e^\alpha(\tau)\in T_A$ and there exists an extension $\rho\succ\tau$ such that $\Phi_e^\alpha(\rho)$ strictly extends $\Phi_e^\alpha(\tau)$, allowing us to compute a path of $T_A$, which is impossible as $A>_M\omega^\omega$.

Begin with the empty string $\sigma_0=\emptyset$.
For $e$ let $D_e$ be the $e$-th dense $\Sigma^0_\alpha$ set of strings.
Given $\sigma_e$, in a hyperarithmetical way, one can find a string $\sigma^\ast_{e}\in D_e$ extending $\sigma_e$.
    Now, by (\ref{disj}) we have a $\Pi^1_1$ function assigning $e$ to the first $\sigma_{e+1}$ extending $\sigma_{e}^\ast$ we find verifying (\ref{disj}). This function is total, and then $\Delta^1_1$. Moreover, it is clear that $\Phi_e(X)$ does not define an element of $A$ for any $X$ extending $\sigma_{e+1}$.
\end{proof}

\subsection{Axiom of choice versus dependent choice}

H.\ Friedman showed that the axiom of $\Sigma^1_1$-dependent choice is strictly stronger than the axiom of $\Sigma^1_1$-choice in the context of second order arithmetic (cf.\ \cite[Corollary VIII.5.14]{SOSOA:Simpson}).
Although the Weihrauch degrees of the principles $\Sigma^1_1\mbox{-}{\sf DC}_{\N^\N}$ and $\Sigma^1_1\mbox{-}{\sf AC}_{\N\to\N^\N}$ are equal (Observation \ref{obs:choice} and Proposition \ref{prop:depecho}), we will see that $\Sigma^1_1\mbox{-}{\sf DC}_{\N}$ is strictly stronger than $\Sigma^1_1\mbox{-}{\sf AC}_{\N\to\N}$, which finally solves Question \ref{que:Borel-choice}:

\begin{Theorem}\label{thm:main-theorem}
$\ac{}<_{\sf W}\dc{}\equiv_{\sf W}\ac{fin}\star\ac{cof}$.
\end{Theorem}

The above result also implies that 
\[\ac{}<_{\sf W}\ac{}\star\ac{}\equiv_{\sf W}\ac{}\star\ac{}\star{}\ac{}.\]

Therefore, Theorem \ref{thm:main-theorem} provides a new natural example of a multivalued function such that the hierarchy of the compositional product with itself stabilizes at the second level.
Another such an example has also been given by \cite{pauly-kihara2-mfcs}.

We now divide Theorem \ref{thm:main-theorem} into two lemmas.

\begin{Lemma}\label{lem:main-separation}
${\sf ATR}_2\not\leq_{\sf W}\ac{}$.
\end{Lemma}

\begin{proof}
Let $A_e$ be the $e$-th computable instance of ${\sf ATR}_2$, that is, $0\fr H\in A(e)$ if and only if $H$ is a jump hierarchy for the $e$-th computable linear order $\prec_e$, and $1\fr p$ if $p$ is an infinite decreasing sequence w.r.t.\ $\prec_e$.
Suppose for the sake of contradiction that $A=\prod_eA_e$ is Medvedev reducible to a homogeneous $\Sigma^1_1$ set $S$.
Let $B$ be the set of all indices $e\in\N$ such that the set of all infinite decreasing sequences w.r.t.\ $\prec_e$ is not Medvedev reducible to $S$, and let $C$ be the set of all indices $e\in\N$ such that the set of all jump-hierarchies for $\prec_e$ is not Medvedev reducible to $S$.
Note that $B$ and $C$ are $\Sigma^1_1$.

Moreover, we claim that $B$ and $C$ are disjoint.
To see this, let $\Phi$ be a continuous function witnessing $A\leq_MS$.
If there is $X\in S$ such that $\Phi(X;0)$ is $i$, then by continuity of $\Phi$, there is a finite initial segment $\sigma$ of $X$ such that $\Phi(Y;0)=i$ for any $Y$ extending $\sigma$.
However, by homogeneity of $S$, $S\cap[\sigma]$ is Medvedev equivalent to $S$.
This means that, for any $e$, $S$ Medvedev bounds either the set of infinite paths or the set of jump-hierarchies for the $e$-th computable tree.
This concludes the claim.

Let ${\sf WO}$ be the set of all indices of well-orderings, and ${\sf NPWO}$ be the set of all indices for computable linear orderings with infinite hyperarithmetic decreasing sequences (i.e., linear orderings which are not pseudo-well-ordered).
Clearly, ${\sf WO}$ is contained in $B$.
Moreover, by H.\ Friedman's theorem \cite{Friedmanthesis} saying that a computable linear order which supports a jump hierarchy cannot have a hyperarithmetical descending sequence (see also Friedman \cite{Fri76} for a simpler proof based on Steel's result \cite{Ste75}), ${\sf NPWO}$ is contained in $C$.
Since $B$ and $C$ are disjoint $\Sigma^1_1$ sets, by an effective version of the Lusin separation theorem (cf.\ \cite[Exercise 4B.11]{MosBook}), there is a $\Delta^1_1$ set $A$ separating $B$ from $C$.
This contradicts Harrington's unpublished result, which states that if a $\Sigma^1_1$ set separates ${\sf WO}$ from ${\sf NPWO}$, then it must be $\Sigma^1_1$-complete (see Goh \cite[Corollary 3]{GohHarrington}).
\end{proof}

\begin{Lemma}\label{lem:cpr-acfincof}
$\dc{}\leq_{\sf W}\ac{fin}\star\ac{cof}$.
\end{Lemma}

\begin{proof}
  Given a pruned $\Sigma^1_1$ tree $T\subseteq\om^{<\om}$, let $f_T$ be the leftmost path through $T$.
  Then $f_T$ has a finite-change higher approximation, i.e., there is a $\Delta^1_1$ sequence approximating $f$ with finite mind-changes (cf.\ \cite{BGM17} for the definition).
  Let $m_T(n)$ be the number of changes of the approximation procedure for $f_T\upto {n+1}$.
  One can assume that $f_T(n)\leq m_T(n)$.
  Then, one can effectively construct a $\Sigma^1_1$ sequence $(S_n)_{n\in\N}$ of cofinite subsets of $\N$ such that $m\in S_n$ implies $m>m_T(n)$.
  In particular, any element $g\in\prod_n S_n$ majorizes $m_T$, and thus $f_T$.
  Use $\ac{cof}$ to choose such a $g$, and consider the $\Sigma^1_1(g)$ tree $T^g=\{\sigma\in T:(\forall n<|\sigma|)\;\sigma(n)<g(n)\}$.
  Then $T^g$ is a finite branching infinite tree since $f_T\in[T^g]$.
  Therefore, as in the proof of Proposition \ref{prop:compact-prouduct-of-two}, one can effectively covert $T^g$ into a $\Sigma^1_1(g)$ infinite binary tree $T^\ast$.
  Use $\Sigma^1_1\mbox{-}{\sf WKL}$ (which is Weihrauch equivalent to $\ac{fin}$, as seen in Observation \ref{obs:choice} and Fact \ref{KMP}) to get an infinite path $p$ through $T^\ast$.
  From $p$ one can easily construct an infinite path through $T$.
\end{proof}

\begin{proof}[Proof of Theorem \ref{thm:main-theorem}]
Clearly, ${\sf ATR}_2\leq_{\sf W}\Sigma^1_1\mbox{-}{\sf C}_{\N^\N}$ since being a jump hierarchy and being an infinite decreasing sequence are arithmetical properties.
Since $\Sigma^1_1\mbox{-}{\sf C}_{\N^\N}$ is Weihrauch equivalent to $\dc{}$ by Proposition \ref{prop:depecho}, we obtain $\ac{}<_{\sf W}\dc{}$ by Lemma \ref{lem:main-separation}.
The equality follows from Lemma \ref{lem:cpr-acfincof}.
\end{proof}

\subsection{Summary of this section}

\begin{Theorem}
  We have \[\ac{fin}\equiv_W\dc{fin}<_W\ac{aof}\equiv_W\dc{aof}<_W\ac{foc}\leq_W\ac{}\]
  and
  \[\ac{cof}\not\geq_W\ac{fin}\qquad\ac{foc}>_W\ac{cof}\not\leq_W\ac{aof}\]
\end{Theorem}

It remains a few questions about $\ac{foc}$:
\begin{Question}
  Is $\ac{foc}<_W\dc{foc}$? Is $\ac{foc}<_W\ac{}$?
\end{Question}

We also do not know if the dependent and independent choice for cofinite sets coincide.
\begin{Question}
  Is $\ac{cof}<_W\dc{cof}$?
\end{Question}

We solved the main question by showing $\ac{}<_W\dc{}$ (Theorem \ref{thm:main-theorem}), but it is just a computable separation.
Therefore, it is natural to ask if $\ac{}$ and $\dc{}$ can be separated even in the hyperarithmetical sense.
In other words, the following is one of the most important open questions, where ${\sf UC}_{\N^\N}$ is the unique choice principle (or equivalently, the choice principle for $\Sigma^1_1$ singletons; cf.\ \cite{KMP}).

\begin{Question}
  Is ${\sf UC}_{\N^\N}\star\ac{}<_W\dc{}$?
\end{Question}

We also ask a question purely on the structure of Medvedev degrees for finite axioms of choice. Define more generally $\ac{P}$ to be $\ac{}$ where the set from which we choose have to be taken from $\sf P$. For instance, if ${\sf P}=\{A\subseteq\N:|A|<\om\}$, then $\ac{P}=\ac{fin}$.

\begin{Question}
  Let ${\sf P}=\{A\subseteq\N:A\subseteq2\}$ and ${\sf Q}=\{A\subseteq\N:|A|\leq2\}$. Is every element of $\ac{Q}$ Medvedev equivalent to some element of $\ac{P}$?
\end{Question}

We are also interested in comparing various kinds of arithmetical transfinite recursion.

\begin{Question}
${\sf ATR}_{\sf 2}\equiv^a_{\sf W}{\sf ATR}_{\sf 2'}\equiv^a_{\sf W}{\sf ATR}_{\sf 2}^{\sf po}$?
\end{Question}

Finally, we mention a few descriptive set theoretic results deduced from our results.

\begin{Theorem}
\begin{enumerate}
\item There is a total analytic set $A$ with compact homogeneous sections such that any total analytic set with compact sections is $\leq_1$-reducible to $A$.
\item For any total analytic set $A$ with closed sections, there is a total analytic set with homogeneous sections which is not $\leq_2$-reducible to $A$.
\item There is a total $F_{\sigma\delta}$ set with $G_\delta$ sections which is not $\equiv_2$-equivalent to any analytic set with closed sections.
\item There is a total closed set which is not $\leq_2$-reducible to any total analytic set with homogeneous sections.
\end{enumerate}
\end{Theorem}

\begin{proof}
(1) follows from the relativization of Theorem \ref{th:fin-maximum}.
(2) follows from the relativization of Theorem \ref{th:upper-density-strong}.
For (3), let $S$ be the set of pairs $(x,y)$ with $y\not\leq_Tx$.
Then $S$ is $F_{\sigma\delta}$, and each $S(x)=\{y:y\not\leq_Tx\}$ is co-countable; hence $G_\delta$.
Suppose that $S$ is $\equiv_2$-equivalent to an analytic set $A$ with closed sections.
In particular, there are $x$-computable functions $h_0,h_1$ such that $S(x)\leq_M^xA(h_0(x))\leq_M^xS(h_1\circ h_0(x))\leq_M^xS(x)$, where $\leq_M^x$ indicates the Medvedev reducibility relative to $x$.
Then we have $S(x)\equiv^x_MA(h_0(x))$.
By relativizing Theorem \ref{thm:lower-density}, there exists a $\Sigma^1_1(x)$ closed set $\om^\om<_M^xB<_M^xS(x)$, which is impossible.
Finally, (4) follows from the relativization of Lemma \ref{lem:main-separation}, and the fact that every $\Sigma^1_1$ set is Medvedev reducible to a closed set in a uniform manner.
\end{proof}

\begin{Ack}
Kihara's research was partially supported by JSPS KAKENHI Grant 17H06738, 15H03634, and the JSPS Core-to-Core Program (A. Advanced Research Networks). Angles d'Auriac's was a Summer Program Fellow of the Japan
Society for the Promotion of Science. He would like to thank JSPS for funding this program.
\end{Ack}

\bibliographystyle{plain}
\bibliography{references}
\end{document}